\documentclass[12pt]{article}
\usepackage{amsmath}
\usepackage{amssymb}
\usepackage{harpoon}
\usepackage{latexsym,color}
\usepackage{ntheorem}
\usepackage{mathrsfs}
\usepackage[numbers,sort&compress]{natbib}

  \allowdisplaybreaks[4]
\title{ \Large  Convergence and non-convergence to 
Bose-Einstein condensation}
\author{ Shuzhe Cai\footnote{Beijing International Center for Mathematical Research, 
Peking University, Beijing 100871, China;  e-mail address: 2206392386@pku.edu.cn}\,\,
and  Xuguang Lu\footnote{Department of Mathematical Sciences,
Tsinghua University, Beijing 100084, China;  e-mail address: xglu@math.tsinghua.edu.cn
} }

\date{}  

\newcommand{\ld}{\lambda}

\newcommand{\p}{\partial}
\newcommand{\vp}{\varphi}
\newcommand{\vep}{\varepsilon}
\newcommand{\og}{\omega}

\newcommand{\sg}{\sigma}

\newcommand{\Gm}{\Gamma}
\newcommand{\dt}{\delta}
\newcommand{\Dt}{\Delta}
\newcommand{\fr}{\frac}

\newcommand{\wh}{\widehat}

\newcommand{\bR}{{\mathbb R}^3 }
\newcommand{\bS}{{\mathbb S}^2 }

\newcommand{\bRR}{{\bR}\times{\bR}}

\newcommand{\bRS}{{\bR}\times {\mathbb S}^2 }
\newcommand{\bRRS}{{\bRR}\times{\mathbb S}^2 }

\newcommand{\la}{\langle}
\newcommand{\ra}{\rangle}
\newcommand{\mR}{{\mathbb R}}
\newcommand{\mN}{{\mathbb N}}

\newcommand{\be}{\begin{myequation}}
\newcommand{\ee}{\end{myequation}}
\newcommand{\bes}{\begin{myeqnarray}}
\newcommand{\ees}{\end{myeqnarray}}
\newcommand{\beas}{\begin{eqnarray*}}
\newcommand{\eeas}{\end{eqnarray*}}
\newcommand{\lb}{\label}
\newcounter{thm}
\theoremseparator{.}
\newtheorem{theorem}{Theorem}[section]

\newtheorem{definition}[theorem]{Definition}
\newtheorem{lemma}[theorem]{Lemma}
\newtheorem{remark}[theorem]{Remark}

\newtheorem{assumption}[theorem]{Assumption}

\newcounter{myequation}[section]

\newenvironment{proof}{{\bf Proof.}}{$\hfill\Box$}
\newenvironment{myequation}{\stepcounter{myequation}\begin{equation}}{\end{equation}}
\newenvironment{myeqnarray}{\stepcounter{myequation}\begin{eqnarray}}{\end{eqnarray}}
\newcommand{\dnumber}{\stepcounter{myequation}}
\begin{document}
\maketitle
\vskip 0.1in \baselineskip 18.2pt
\begin{abstract}
The paper is a continuation of our previous work on the strong convergence to equilibrium 
for the spatially homogeneous
Boltzmann equation for Bose-Einstein particles for isotropic solutions at low temperature.
Here we study the influence of the particle interaction potentials on the convergence to Bose-Einstein condensation (BEC).
Consider two cases of certain potentials
that are such that the corresponding scattering cross sections are bounded and 1) have a lower bound
${\rm const.}\min\{1, |{\bf v-v}_*|^{2\eta}\}$ with ${\rm const}.>0, 0\le \eta<1$, 
and 2) have an upper bound ${\rm const.}\min\{1, |{\bf v-v}_*|^{2\eta}\}$ with $\eta\ge 1$.
For the first case, the long time convergence to BEC i.e.
$\lim\limits_{t\to\infty}F_t(\{0\})=F_{\rm be}(\{0\})$ is proved for a class of initial data 
having very low temperature and thus it holds the strong convergence 
to equilibrium.
For the second case we show that if initially $F_0(\{0\})=0$, then 
$ F_t(\{0\})=0$ for all $t\ge 0$ and thus there is no convergence to BEC hence no strong convergence 
to equilibrium.


\end{abstract}

\begin{center}\section { Introduction}\end{center}

We study the space homogeneous quantum Boltzmann equations for Bose-Einstein particles which (after normalizing some physics constants) is given by 
 \be\fr{\partial}{\partial t}f({\bf v},t)=\int_{{\bRS}}B({\bf
{\bf v-v}_*},\og)\big(f'f_*'(1+f)(1+f_*)-ff_*(1+f')(1+f_*')\big) {\rm d}\omega{\rm
d}{\bf v_*}\label{Equation1}\ee
with $({\bf v}, t)\in{\mathbb R}^3\times(0,\infty)$, where the solution $f=f({\bf v},t)\ge 0$ is the  number density of particles at time $t$ with the velocity ${\bf v}$; 
$f_*=f({\bf v_*},t), f'=f({\bf v'},t),f_*'=f({\bf v_*'},t)$, and ${\bf v},{\bf v_*}$ and ${\bf v'},{\bf v_*'}$ are velocities of two particles before and after
their collision:
\be{\bf v}'={\bf v}- (({\bf v}-{\bf v}_*)\cdot\omega)\omega,\quad {\bf
v}_*'={\bf v}_*+ (({\bf v}-{\bf v}_*)\cdot\omega)\omega, \qquad \omega\in{\mathbb S}^2.
\lb{colli}\ee
Eq.(\ref{Equation1}) as well as that for Fermi-Dirac particles were first derived by Nordheim \cite{Nordheim} and Uehling $\&$ Uhlenbeck \cite{Uehling and Uhlenbeck} and then developed by \cite{weak-coupling},\cite{Chapman and Cowling},\cite{Do},\cite{ESY},\cite{EMV0},\cite{Lions},\cite{Lu2000},\cite{LS}. 
A physical explanation for deriving Eq.(\ref{Equation1}) in  Chp.17 of \cite{Chapman and Cowling} is as follows: 
when the mean distance between neighbouring molecules is comparable with the size of the quantum wave fields
with which molecules are surrounded, a state of congestion is formed. For a gas
composed of Bose-Einstein particles, according to quantum theory,
the presence of a like particle in the velocity-range ${\rm d}{\bf v}$ increases the
probability that a particle will enter that range; the presence of $f({\bf v}){\rm d}{\bf v}$
particles per unit volume increases this probability in the ratio $1+
f({\bf v})$. This yields Eq.(\ref{Equation1}). 
 Symmetrically, replacing the `` increases " and $``1+f({\bf v})" $ with
``decreases" and  $``1- f({\bf v})"$ (Pauli exclusion principle) leads to the Boltzmann
equation for Fermi-Dirac particles.
We are interested in the quantum effect of the solutions of Eq.(\ref{Equation1}). Recent research have shown that 
Eq.(\ref{Equation1}) is a suitable mesoscopic model for describing the evolution of Bose-Einstein condensation (BEC), see e.g.
\cite{AN1},\cite{CL},\cite{EV1},\cite{EV2},\cite{Lu2004},\cite{Lu2016},\cite{Nouri},\cite{SH}.
Here we show further that the interaction potential has a significant impact on the occurrence and convergence of BEC.

The function $B({\bf {\bf v-v}_*},\omega)$ in (\ref{Equation1}) is the collision kernel which we assume takes the
form
 \be B({\bf {\bf v-v}_*},\omega)= \fr{1}{(4\pi)^2}|({\bf v-v}_*)\cdot\omega|
\Phi(|{\bf v}-{\bf v}'|, |{\bf {\bf v}-{\bf v}_*'}|)\lb{kernel}\ee where
\be 0\le \Phi\in C_b({\mR}_{\ge 0}^2),\quad  \Phi(r,\rho)=\Phi(\rho, r)\quad \forall\, (r,\rho)\in{\mR}_{\ge 0}^2 \lb{Phi}.\ee
According to \cite{weak-coupling} and \cite{ESY},  in the weak-coupling regime the function $\Phi$ in (\ref{kernel}) is
given by (after normalizing physical parameters)
\be \Phi(r,\rho)=\big(\widehat{\phi}(r)+\widehat{\phi}(\rho)\big)^2,\quad r,\rho\ge 0\label{kernel2}\ee
where $\widehat{\phi}$ is the (generalized) Fourier transform of a radially symmetric particle interaction potential
$ \phi(|{\bf x}|)\in {\mR}$:
$$\widehat{\phi}(r):=
\int_{{\bR}}\phi(|{\bf x}|) e^{-{\rm i}\xi\cdot{\bf x}}{\rm d}{\bf x} \Big|_{|\xi|=r}.$$
In particular if $\phi(|{\bf x}|)=\frac{1}{2}\delta({\bf x})$, where $\delta({\bf x})$ is the three dimensional Dirac delta function concentrating at
${\bf x}=0$, then  $\widehat{\phi}\equiv \fr{1}{2}$ hence $\Phi\equiv 1$ and (\ref{kernel}) becomes the hard sphere model:
\be B({\bf
{\bf v-v}_*},\omega)=\frac{1}{(4\pi)^2}|({\bf v-v}_*)\cdot\omega|.\lb{Hard}\ee 
In view of physics, the hard sphere model
(\ref{Hard})
can be extended to a more practical case where the interaction potential $\phi({\bf x})$ contains an attractive  term $\fr{-1}{2}U(|{\bf x}|)$, i.e.
$\phi(|{\bf x}|)=\fr{1}{2}(\delta({\bf x})-U(|{\bf x}|))$ where $U(|{\bf x}|)\ge 0$ on ${\bR}$.
The generalized Fourier transform of $\phi$ is
$\widehat{\phi}(r)=\fr{1}{2}\big(1-\widehat{U}(\xi)|_{|\xi|=r}\big)$.
For any $0<\eta<1$, a nonnegative potential $U(|{\bf x}|)$ satisfying 
$\widehat{U}(\xi)|_{|\xi|=r}= \fr{1}{1+r^{\eta}}$
can be constructed (see e.g. Appendix in \cite{CL}); in this case we have
\be \widehat{\phi}(r)= \fr{1}{2}\cdot\fr{r^{\eta}}{1+r^{\eta}}
,\quad r\ge 0.\lb{phi}\ee
For the case $\eta\ge 1$, for instance $\eta=2$, one may take $U$ to be the Yukawa potential $U(|{\bf x}|)=\fr{1}{4\pi |{\bf x}|}e^{-|{\bf x}|}$ so that 
$\widehat{U}(\xi)|_{|\xi|=r}= \fr{1}{1+r^{2}}$ and so the generalized Fourier transform of $\phi(|{\bf x}|)=\fr{1}{2}(\delta({\bf x})-U(|{\bf x}|))$  is 
\be \widehat{\phi}(r)=\fr{1}{2}\cdot \fr{r^2}{1+r^2},\quad r\ge 0.\lb{Yukawa}\ee

Due to the strong nonlinear structure  of the collision integrals in Eq.(\ref{Equation1}) 
and the effect of condensation,
existence and uniqueness of solutions of Eq.(\ref{Equation1}) for anisotropic initial data
have been so far only proven for finite time interval without smallness assumption on the initial data
\cite{Briant-Einav}
and  for global time interval with a smallness assumption on the initial data \cite{LiLu}, 
and for the case where initial data are closed to equilibrium \cite{AN2}, \cite{OW}, \cite{Zhou}.

For global in time solutions with general initial data at low temperature, so far we could only study 
the following weak form of Eq.(\ref{Equation1})  for isotropic solutions:
\bes&&\fr{{\rm d}}{{\rm d}t}\int_{{\bR}}\vp(|{\bf v}|^2/2)f(|{\bf v}|^2/2,t){\rm d}{\bf v}\nonumber\\
&&
=\fr{1}{2}\int_{{\bRRS}}(\vp+\vp_*-\vp'-\vp'_*)B({\bf {\bf v-v}_*},\omega)f'f_*'{\rm d}{\bf v}{\rm d}{\bf v}_*
{\rm d}\og  \qquad \nonumber\\
&&+\int_{{\bRRS}}(\vp+\vp_*-\vp'-\vp'_*)B({\bf {\bf v-v}_*},\omega)ff'f_*'{\rm d}{\bf v}{\rm d}{\bf v}_*
{\rm d}\og  \qquad \label{weak}\ees
for all test functions $\vp$ and all $t\in [0,\infty)$. Here the common
quartic terms $f'f_*'ff_*,\, ff_*f'f_*'$ cancel each other. By changing variables
$x=|{\bf v}|^2/2, y=|{\bf v}'|^2/2, z=|{\bf v}_*'|^2/2$ and
letting the measure $F_t\in {\cal B}_1^{+}({\mR}_{\ge 0})$ be defined by ${\rm d}F_t(x)=f(x,t)\sqrt{x}{\rm d}x$,
 the collision integrals in (\ref{weak}) can be rewritten
(see\cite{CL})
$$
\fr{1}{2}\int_{{\bRRS}}(\vp+\vp_*-\vp'-\vp'_*)Bf'f_*'
{\rm d}{\bf v}{\rm d}{\bf v}_*
{\rm d}\og=4\pi\sqrt{2}\int_{{\mathbb R}_{\ge 0}^2}{\cal J}[\varphi]{\rm d}F_t(y){\rm d}F_t(z),$$
$$
\int_{{\bRRS}}(\vp+\vp_*-\vp'-\vp'_*)B ff'f_*'
{\rm d}{\bf v}{\rm d}{\bf v}_*
{\rm d}\og
=4\pi\sqrt{2}\int_{{\mathbb R}_{\ge 0}^3}{\cal K}[\varphi]{\rm d}F_t(x){\rm d}F_t(y){\rm d}F_t(z)$$
where $B=B({\bf {\bf v-v}_*},\omega)$ is given by (\ref{kernel}) with
(\ref{Phi}),
${\cal J}[\varphi],{\cal K}[\varphi]$ are linear operators of $\vp\in C_b^2({\mR}_{\ge 0})$ 
defined in (\ref{JKW})-(\ref{Y}) below. This allows the solutions of Eq.(\ref{Equation1}) to be positive Borel measures
so that the evolution of condensation for low temperature can be investigated. 
The measure-valued isotropic solutions of Eq.(\ref{Equation1}) in the weak form is defined as follows (\cite{CL}):
\vskip2mm

\begin{definition}\label{definition1.1} Let $B({\bf {\bf v-v}_*},\omega)$ be given by (\ref{kernel}), (\ref{Phi}).
Let $F_0\in {\cal B}_{1}^{+}({\mathbb R}_{\ge 0})$. We say that a
family $\{F_t\}_{t\ge 0}\subset {\cal B}_{1}^{+}({\mathbb R}_{\ge 0})$, or simply $F_t$, is a conservative  measure-valued isotropic solution of Eq.(\ref{Equation1}) on the time-interval $[0, \infty)$ with the initial datum $F_t|_{t=0}=F_0$ if

{\rm(i)} $N(F_t)=N(F_0),\,\, E(F_t)=E(F_0)$ for all $t\in [0, \infty)$,

{\rm(ii)} for every $\varphi\in C^2_b({\mathbb R}_{\ge 0})$,  $t\mapsto \int_{{\mathbb R}_{\ge 0}}\varphi(x){\rm d}F_t(x)$ belongs to
$C^1([0, \infty))$, and

{\rm(iii)} for every $\varphi\in C^2_b({\mathbb R}_{\ge 0})$
\be\frac{{\rm d}}{{\rm d}t}\int_{{\mathbb R}_{\ge 0}}\varphi{\rm
d}F_t= \int_{{\mathbb R}_{\ge 0}^2}{\cal J}[\varphi]{\rm d}^2F_t+ \int_{{\mathbb R}_{\ge 0}^3}{\cal K}[\varphi]{\rm d}^3F_t\qquad \forall\,t\in[0, \infty). \label{Equation3}\ee
Here ${\rm d}^2F={\rm d}F(y){\rm d}F(z), {\rm d}^3F={\rm d}F(x){\rm d}F(y){\rm d}F(z).$
\end{definition}
\vskip2mm

{\bf Notation.}
Let ${\cal B}_k({\mR}_{\ge 0})$ ($k\ge 0$) be the linear space
of signed real Borel measures $\mu$ on ${\mR}_{\ge 0}$ satisfying
$\int_{{\mR}_{\ge 0}}(1+|x|)^k{\rm d}|\mu|(x)<\infty$,  where
$|\mu|$ is the total variation of $\mu$. 
Define 
$$\|\mu\|=\|\mu\|_0,\quad \|\mu\|_k=\int_{{\mR}_{\ge 0}}(1+|x|)^k{\rm d}|\mu|(x).$$
Let
${\cal B}_k^{+}({\mR}_{\ge 0})=\{F \in {\cal B}_k({\mR}_{\ge 0})\,|\, F \ge 0\}.$ For the case $k=0$ we also denote
${\cal B}({\mR}_{\ge 0})={\cal B}_{0}({\mR}_{\ge 0}), {\cal B}^{+}({\mR}_{\ge 0})={\cal B}_{0}^{+}({\mR}_{\ge 0})$.

Let $C^k_b({\mathbb R}_{\ge 0})$ with $k\in {\mathbb N}$ be the class of bounded continuous functions on
${\mathbb R}_{\ge 0}$ having bounded continuous derivatives on ${\mathbb R}_{\ge 0}$ up to the order $k$. Let
$$
C^{1,1}_b({\mathbb R}_{\ge 0})=\Big\{ \varphi\in C^1_b({\mathbb R}_{\ge 0})\, \Big|\,
\,
\frac{{\rm d}}{{\rm d}x}\varphi\in {\rm Lip}({\mathbb R}_{\ge 0})\Big\}$$
where ${\rm Lip}({\mathbb R}_{\ge 0})$ is the class of functions
satisfying Lipschitz condition on ${\mathbb R}_{\ge 0}$. For any 
$\vp\in 
C^{1,1}_b({\mathbb R}_{\ge 0})$, let
\be{\cal J}[\varphi](y,z)=\frac{1}{2}
\int_{0}^{y+z}{\cal K}[\varphi](x,y,z)
\sqrt{x}{\rm d}x,\quad {\cal K}[\varphi](x, y,z)=W(x,y,z)\Delta\varphi(x,y,z)
,\lb{JKW}\ee
\be \Delta\varphi(x,y,z)=\varphi(x)+\varphi(x_*)-\varphi(y)-\varphi(z)=
(x-y)(x-z)
\int_{0}^1\!\!\!\int_{0}^{1}\vp''(\xi) {\rm d}s {\rm
d}t
\lb{diff}\ee 
$$\xi=y+z-x+t(x-y)+s(x-z),\quad x,y,z\ge 0,\, x_*=(y+z-x)_{+},$$
\be
W(x,y,z)=\fr{1}{4\pi\sqrt{xyz}}
\int_{|\sqrt{x}-\sqrt{y}|\vee |\sqrt{x_*}-\sqrt{z}|}
^{(\sqrt{x}+\sqrt{y})\wedge(\sqrt{x_*}+\sqrt{z})}{\rm d}s
\int_{0}^{2\pi}\Phi(\sqrt{2}s, \sqrt{2} Y_*){\rm d}\theta
\qquad {\rm if}\quad x_*xyz>0,\lb{W1}\ee
\be
 W(x,y,z)=\left\{\begin{array}
{ll}
 \displaystyle
\fr{1}{\sqrt{yz}}
\Phi(\sqrt{2y}, \sqrt{2z}\,)
\qquad\,\, \qquad\qquad{\rm if}\quad  x=0,\,y>0,\, z>0 \\ \\  \displaystyle
\fr{1}{\sqrt{xz}}
\Phi(\sqrt{2x}, \sqrt{2(z-x)}\,)
\qquad \qquad {\rm if}\quad  y=0,\, z>x>0
 \\ \\ \displaystyle
\fr{1}{\sqrt{xy}}
\Phi(\sqrt{2(y-x)}, \sqrt{2x}\,)
\qquad \quad \quad{\rm if}\quad  z=0,\, y>x>0
 \\ \\ \displaystyle
0\qquad \qquad \qquad \qquad \qquad \qquad \qquad \quad {\rm others}\end{array}\right.\lb{W2}\ee

\be Y_*=Y_*(x,y,z,s,\theta)=\left\{\begin{array}
{ll}\displaystyle
\bigg|\sqrt{\Big(z-\fr{(x-y+s^2)^2}{4s^2}\Big)_{+}
}
+e^{{\rm i}\theta}\sqrt{\Big(x-\fr{(x-y+s^2)^2}{4s^2}
\Big)_{+}}\,\bigg|\quad {\rm if}\quad s>0\\ \\ \displaystyle
0\qquad\qquad  {\rm if}\quad s=0
\end{array}\right.\lb{Y}\ee
where $\Phi(r,\rho)$ is given in  (\ref{Phi}),$(u)_{+}=\max\{u, 0\}$,
$a\vee b =\max\{a,b\},\, a\wedge b =\min\{a,b\},$
${\rm i}=\sqrt{-1}$.
Note that according to (\ref{W2}), if $x_*=0$ i.e. if $x\ge y+z$, then $W(x,y,x)=0$. 
It is easily checked that the following implication holds true:
\be \left\{\begin{array}
		{ll}\displaystyle s>0\,, x_*>0\,,
 |\sqrt{x}-\sqrt{y}|\vee |\sqrt{x_*}-\sqrt{z}| \le s\le (\sqrt{x}+\sqrt{y})\wedge (\sqrt{x_*}+\sqrt{z})
\\  \displaystyle
\\ \Longrightarrow\quad  x-\fr{(x-y+s^2)^2}{4s^2}\ge0,\quad   z-\fr{(x-y+s^2)^2}{4s^2}\ge 0.
	\end{array}\right.\lb{KK0}\ee
The transition from
(\ref{W1}) to (\ref{W2}) in defining $W$ is due to
 the identity (in case $x_*>0$)
\be (\sqrt{x}+\sqrt{y})\wedge (\sqrt{x_*}+\sqrt{z})-|\sqrt{x}-\sqrt{y}|\vee |\sqrt{x_*}-\sqrt{z}|
=2\min\{\sqrt{x},\sqrt{x_*},\sqrt{y},\sqrt{z}\}\lb{1.difference}\ee
 from which one sees also that if $\Phi(r,\rho)\equiv 1$  then $W(x,y,z)$ becomes the function corresponding to the hard sphere model:
\be W_H(x,y,z)=\fr{1}{\sqrt{xyz}}\min\{\sqrt{x},\sqrt{y},\sqrt{z},\sqrt{x_*}\}, \qquad  x_*xyz>0\lb{hard}\ee
and for other cases of $x,y,z\in {\mR}_{\ge 0}$, $W_H(x,y,z)$ is given by (\ref{W2}) with $\Phi\equiv 1$.

As has been proven in \cite{Lu2013} that the test function
space $C_b^2({\mR}_{\ge 0})$ in Definition \ref{definition1.1}  can be weaken to
$C^{1,1}_b({\mathbb R}_{\ge 0})$.  In fact, for the function $\Phi$
satisfying (\ref{Phi}), it is
not difficult to prove that for any $\vp\in C^{1,1}_b({\mathbb R}_{\ge 0})$
the functions $(y,z)\mapsto (1+\sqrt{y}+\sqrt{z})^{-1}{\cal J}[\vp](y,z), (x,y,z)\mapsto {\cal K}[\vp](x,y,z)$ belong to $C_b({\mR}_{\ge 0}^2)$ and $C_b({\mR}_{\ge 0}^3)$
respectively.
Thus there is no problem of integrability in the right hand side of Eq.(\ref{Equation3}).
 A typical example of $\vp\in C^{1,1}_b({\mathbb R}_{\ge 0})$ is $\vp_{\vep}(x)=[(1-x/\vep)_{+}]^2$ (with $\vep>0$) which is 
 very helpful in the study of convergence to BEC.

\vskip2mm

{\bf Kinetic Temperature.}   Let $F\in {\mathcal B}_{1}^{+}({\mathbb R}_{\ge 0})$, 
$$N=N(F)=\int_{{\mR}_{\ge 0}}{\rm d}F(x),\quad  E=E(F)=\int_{{\mR}_{\ge 0}}x{\rm d}F(x)$$ 
and suppose $N>0$. If $m$ is the mass of one particle,
then  $m4\pi\sqrt{2}N$, $m 4\pi\sqrt{2}E$ are total
mass and kinetic energy of the particle system per unite
space volume. Keeping in mind the constant $m4\pi\sqrt{2}$,
there will be no confusion if we also call $N$ and $E$ the
mass and energy of a particle system.
The kinetic temperature $\overline{T}$ and the kinetic
critical temperature $\overline{T}_c$ are defined by (see e.g.\cite{Lu2004} and references therein)
$$\overline{T}=\frac{2m}{3k_{\rm B}}
\frac{E}{N}
,\qquad\overline {T}_c=\frac{\zeta(5/2)}{(2\pi)^{1/3}[\zeta(3/2)]^{5/3}}
\frac{2m}{k_{\rm B}}N^{2/3} $$
where $k_{\rm B}$ is the Boltzmann constant, $\zeta(s)=\sum_{n=1}^{\infty} n^{-s}, s>1$. Temperature effects are often expressed by a function of the ratio
$$ \frac{\overline{T}}{\overline {T}_c}= \frac{(2\pi)^{1/3}[\zeta(3/2)]^{5/3}}{3\zeta(5/2)}
\frac{E}{N^{5/3}}= 2.2720\frac{E}{N^{5/3}}. $$

\vskip2mm

{\bf Regular-Singular Decomposition.}  According to measure theory (see e.g.\cite{Rudin}), every
finite positive Borel measure can be uniquely decomposed into the
regular part and then singular part with respect to the (weighted) Lebesgue measure. For instance
if $F\in {\cal B}_1^{+}({\mathbb R}_{\ge 0})$, then
there exist unique $0\le f\in L^1({\mathbb R}_{\ge 0},(1+x)\sqrt{x}{\rm d}x)$, $\nu\in{\cal B}_1^{+}({\mathbb R}_{\ge 0})$ and a Borel set $Z\subset {\mathbb R}_{\ge 0}$
such that
$${\rm d}F(x)=f(x)\sqrt{x}{\rm d}x+{\rm d}\nu(x),\quad mes(Z)=0,\quad \nu({\mR}_{\ge 0}\setminus Z)=0.$$
$f$ and $\nu$ are called the regular part and the singular part of $F$
respectively. $F$ is called  regular if $\nu=0$, and singular if 
$\nu\neq 0$ and $\int_{{\mathbb R}_{\ge 0}}f(x)\sqrt{x}\,{\rm d}x=0$. 
\vskip2mm

{\bf Bose-Einstein Distribution.}  According to Theorem 5 of \cite{Lu2004} and
its equivalent version proved in the Appendix of \cite{Lu2013} we know that
for any  $N>0$, $E>0$ the Bose-Einstein distribution $F_{{\rm be}}\in {\mathcal B}_1^{+}({\mathbb R}_{\ge 0})$ given by
\beas {\rm d}F_{\rm be}(x)=f_{\rm be}(x)\sqrt{x}{\rm d}x+\big(1-(\overline{T}/\overline{T}_c)^{3/5}\big)_{+}
N\dt(x){\rm d}x\eeas
is the unique equilibrium solution of Eq.(\ref{Equation3})
satisfying $N(F_{\rm be})=N, E(F_{\rm be})=E$,  where
\be f_{\rm be}(x)=
\left\{\begin{array}{ll}\displaystyle
\frac{1}{Ae^{x/\kappa}-1},\quad A>1,\,\,\qquad {\rm if}
\quad \overline{T}/\overline{T}_c>1,\\ \\ \displaystyle
\frac{1}{e^{x/\kappa}-1},\qquad A=1\,\,\,\qquad  {\rm if}\quad
\overline{T}/ \overline{T}_c\le 1
\end{array}\right.\label{1-E}\ee
$\dt(x)$ is the Dirac
delta function concentrated at $x=0$, i.e.
$\dt(x){\rm d}x={\rm d}\nu_0(x)$  where  $\nu_0$ is the Dirac measure
concentrated at $x=0$, and
functional relations of the coefficients $A=A(N,E)\ge 1, \kappa=\kappa(N,E)>0$
can be found in for instance Proposition 1 in \cite{Lu2005}. From (\ref{1-E}) one sees that
$\overline{T}/\overline{T}_c\ge 1 \, \Longrightarrow\,  {\rm d}F_{\rm be}(x)=f_{\rm be}(x)\sqrt{x}{\rm d}x$, and
$$\overline{T}/\overline{T}_c<1\, \Longleftrightarrow\, F_{{\rm be}}(\{0\})=\big(1-(\overline{T}/\overline{T}_c)^{3/5}\big)N>0.$$
The positive number $(1-(\overline{T}/\overline{T}_c)^{3/5} )N$
is called the Bose-Einstein Condensation (BEC) of the equilibrium state of Bose-Einstein particles at low temperature $\overline{T}<\overline{T}_c$.
\vskip2mm

{\bf Entropy.}  The entropy functional for Eq.(\ref{Equation1}) is
\be S(f)=\int_{{\bR}}\big((1+f({\bf v}))\log(1+f({\bf v}))
-f({\bf v})\log f({\bf v})\big){\rm d}{\bf v},\quad 0\le f\in L^1_2({\bR}).\lb{ent}\ee
where
$$ L^1_s({\bR})=\Big\{f\in L^1({\bR})\,\,\Big|\,\,
\|f\|_{L^1_s}:=\int_{{\bR}}\la {\bf v}\ra^s|f({\bf v})|{\rm d}{\bf v}<\infty\Big\},\quad \la {\bf v}\ra:=(1+|{\bf v}|^2)^{1/2}.$$
For functions $0\le f\in L^1({\mathbb R}_{\ge 0},(1+x)\sqrt{x}\,{\rm d}x)$, the entropy
$S(f)$ is defined by $S(f)=S(\bar{f})$
with $\bar{f}({\bf v}):=f(|{\bf v}|^2/2)$, i.e.
\be S(f)=4\pi\sqrt{2}\int_{{\mathbb R}_{\ge 0}}\big((1+f(x))\log(1+f(x))-f(x)\log f(x)\big)\sqrt{x}\,{\rm d}x.\lb{ent3}\ee
As in\cite{CL}, the entropy $S(F)$ of a measure
$F\in {\cal B}_1^{+}({\mR}_{\ge 0})$ can be defined by
\be S(F):=\sup_{\{f_n\}_{n=1}^{\infty}}\limsup_{n\to\infty}S(f_n)\lb{ent1}\ee
where $\{f_n\}_{n=1}^{\infty}\subset L^1({\mathbb R}_{\ge 0},(1+x)\sqrt{x}\,{\rm d}x)$ satisfying
\bes&& f_n\ge 0,\quad \sup_{n\ge 1}\int_{{\mR}\ge 0}(1+x)f_n(x)\sqrt{x}{\rm d}x<\infty;  \lb{1.20}\\
&&\lim_{n\to\infty}\int_{{\mR}_{\ge 0}}\vp(x)f_n(x)\sqrt{x}{\rm d}x
=\int_{{\mR}_{\ge 0}}\vp(x){\rm d}F(x)\qquad \forall\,\vp\in C_b({\mR}_{\ge 0}).\dnumber\lb{1.21}\ees
It has been proven (see Lemma 3.2 in \cite{CL}) that  if $0\le f\in L^1({\mathbb R}_{\ge 0},(1+x)\sqrt{x}\,{\rm d}x)$ is the regular part of $F\in B_1^{+}
({\mR}_{\ge 0})$, i.e. if ${\rm d}F(x)=f(x)\sqrt{x}\,{\rm d}x+{\rm d}\nu(x)$  with the singular part $\nu$,
then
\be S(F)=S(f)\lb{ent2}\ee
which indicates that 1) the singular part $\nu$ of $F$ has no contribution to the entropy $S(F)$,
2) $F$ is not singular if and only if $S(F) > 0.$

\noindent Although the entropy $S(F_t)$ does not provide any information about the singular part of $F_t$, the entropy difference $S(F_{\rm be})-S(F_t)$ does control the convergence to equilibrium in a semi-strong norm
(see (\ref{2.10})), and thus (together with the Assumption \ref{assp} below) the entropy does help to prove the convergence $\lim\limits_{t\to \infty}F_t(\{0\})=F_{\rm be}(\{0\})$ hence to prove the strong convergence $\lim\limits_{t\to \infty}\|F_t-F_{\rm be}\|_1=0$ (see Lemma \ref{lemma2.1} and the proof of Theorem \ref{theorem1.2}).
 \vskip2mm

 Our assumptions on the collision kernel (\ref{kernel}) are as follows that 
imply different behaviour of solutions of Eq.(\ref{Equation1}):
 \vskip2mm

\begin{assumption}\lb{assp}    $B({\bf {\bf v-v}_*},\omega)$ is given by
(\ref{kernel}),(\ref{Phi}) where $\Phi$ also satisfies
\bes&& \rho\mapsto \Phi(r,\rho)\,\,\,{\rm is\,\,non-decreasing\,\,on}\,\,{\mR}_{\ge 0}\qquad \forall\, r\in{\mR}_{\ge 0}, \lb{Phi*}\\
&& b_0\min\{1,\,(r^2+\rho^2)^{\eta}\}\le \Phi(r,\rho)\le 1\qquad \forall\,
(r,\rho)\in{\mR}^2_{\ge 0} \dnumber \lb {1.6}\ees
for some constants $0<b_0<1, \eta\ge 0.$
\end{assumption}
\vskip2mm

\begin{assumption}\lb{assp*}    $B({\bf {\bf v-v}_*},\omega)$ is given by
(\ref{kernel}),(\ref{Phi}) where $\Phi$ also satisfies 
$\Phi\in C_b^2({\mR}_{\ge 0})$ and
\be
0 \le \Phi(r,\rho)\le \min\{1,\,(r^2+\rho^2)^{\eta}\}\qquad \forall\,
(r,\rho)\in{\mR}^2_{\ge 0} \lb {1.6*}\ee
for some constant $\eta\ge 0$.
\end{assumption}
\vskip2mm
Recall that in the scattering theory the function $\Phi(|{\bf v}-{\bf v}'|, |{\bf {\bf v}-{\bf v}_*'}|)$ in 
(\ref{kernel}) is the scattering cross section. Since $ |{\bf v-v}'|^2+|{\bf v-v}_*'|^2=|{\bf v-v}_*|^2$, the
assumptions (\ref{1.6}),  (\ref{1.6*}) can be directly written as
\be b_0\min\{1,\, |{\bf v-v}_*|^{2\eta}\}\le \Phi(|{\bf v-v}'|, |{\bf v-v}_*'|)\le 1,\lb{1.29}\ee
\be 0 \le \Phi(|{\bf v-v}'|, |{\bf v-v}_*'|)\le \min\{1,\, |{\bf v-v}_*|^{2\eta}\}.\lb{1.30}\ee
If $\Phi$ is given by (\ref{kernel2}), 
then the Assumption \ref{assp} with $0<\eta<1$ is satisfied for the interaction model (\ref{phi}), and   
the Assumption \ref{assp*} with $\eta=2$  is satisfied for the interaction model (\ref{Yukawa}).
Note that the interaction potential $\phi(|{\bf x}|)=\fr{1}{2}(\delta({\bf x})-U(|{\bf x}|))$ is not Lebesgue integrable on ${\bR}$. 
For the case where $\phi$ is Lebesgue integrable on ${\bR}$, it is easily seen that a sufficient condition for 
 $\Phi$ satisfying the Assumption \ref{assp*} with $\eta=2$ is that $\phi$ is sufficiently balanced:
\be 
\int_{\mR^3}(1+|{\bf x}|^2)|\phi(|{\bf x}|)|{\rm d}{\bf x}<\infty,\quad 
\int_{\mR^3}\phi(|{\bf x}|){\rm d}{\bf x}=0.\lb{banalced}\ee
It should be noted that each of the assumptions (\ref{Phi}), (\ref{Phi*})-(\ref{1.6*}) on $\Phi$  (given by (\ref{kernel2}))  excludes the inverse power law model: $\wh{\phi}(r)={\rm const}. r^{-3+\alpha}$  with $0<\alpha<3$, which comes from the inverse power law potential $\phi(|{\bf x}|)={\rm const}. |{\bf x}|^{-\alpha}$. For this inverse power law model, so far 
there has been no result on the condensation for solutions of Eq.(\ref{Equation1}); and 
the well-posednees of the equation (with initial data closing to equilibrium)
has been just proven in \cite{Zhou}.
\vskip2mm 

{\bf Main Result.} The main result of the paper is as follows:

\begin{theorem}\lb{theorem1.2} Let $F_0\in {\mathcal B}_1^{+}({\mathbb R}_{\ge 0})$ with
$N=N(F_0)>0, E=E(F_0)>0$ and let $F_{\rm be}$ be
the unique Bose-Einstein distribution with the same mass $N$ and energy $E$. 

(I) Let $B({\bf {\bf v-v}_*},\omega)$ satisfy the Assumption \ref{assp} with $0\le \eta<1$. Let 
$0\le\alpha<1-\eta, \beta=\fr{1}{2}(1-\eta-\alpha), p=\fr{3}{2}+\alpha$, and suppose 
\be \fr{E}{N}\le \fr{1}{4p},\quad \inf_{0<\vep\le 1/2}\fr{F_0([0,\vep])}{\vep^{\alpha}}\ge 
\fr{2^{9.5+\alpha}(4p)^{2\beta}}{(\log 2)b_0}
\Big(\fr{p}{\beta}\Big)^{2p}\Big(\fr{E}{N}\Big)^{\fr{1}{2}+2\beta}.
\lb{local}\ee
Let $F_t\in {\cal B}_{1}^{+}({\mR}_{\ge 0})$ be a conservative measure-valued isotropic solution of Eq.(\ref{Equation1}) with 
the initial datum $F_0$ obtained in Theorem \ref{theorem2.2} with $\fr{1}{20}<\ld<\fr{1}{19}$. Then 
$\overline{T}/\overline{T}_c<<1$ and 
\be \big|F_t(\{0\})-F_{\rm be}(\{0\})\big|\le \|F_t-F_{{\rm be}}\|_1\le C(1+t)^{-\fr{(1-\eta)\ld}{2(4-\eta)}}\qquad \forall\, t\ge 0.\lb{conv}\ee
Here the constant $0<C<\infty$ depends only on $N,E, b_0, \eta$ and $\ld.$

(II) Let $B({\bf {\bf v-v}_*},\omega)$ satisfy the Assumption \ref{assp*} with $\eta\ge 1$.
Then for any conservative  measure-valued isotropic solution $F_t\in B_1^{+}({\mR}_{\ge 0})$ of
Eq.(\ref{Equation1}) on $[0,\infty)$ with the initial datum $F_0$, it holds
\be F_t(\{0\}) \le e^{at} F_0(\{0\})\qquad \forall\,t\ge 0 \lb{noBEC}\ee 
with $a =48\sqrt{2} N^2.$
In particular if $F_0(\{0\})=0$ and $\overline{T}/\overline{T}_c<1$, then $F_t(\{0\})\equiv 0$ and 
\be \inf_{t\ge 0}\|F_t-F_{\rm be}\|\ge F_{\rm be}(\{0\})>0.\lb{non-conv}\ee
\end{theorem}
\vskip2mm

\begin{remark}\lb{remark1.5}
\rm (1) In part (I) of Theorem\ref{theorem1.2}, if $0\le \eta<1/4$, our previous work \cite{CL} 
shows that under the only low temperature condition 
$\overline{T}/\overline{T}_c<1$ the strong convergence (\ref{conv}) holds  for all
initial data $F_0$.  In \cite{CL}, a key step of proving (\ref{conv}) is that if $\alpha=\fr{1}{10}(1-4\eta)>0$
then an iterative scheme can be constructed to create a local property which is similar to the second condition in (\ref{local}) 
where $F_0$ is replaced by $F_{t_0}$ for
a solution $F_{t}$ at some time $t_0>0$, and this together with $\overline{T}/\overline{T}_c<1$ leads to 
$\inf\limits_{t\ge t_1}F_{t}(\{0\})>0$ for some $t_1>t_0$.
When $\eta\ge 1/4$, such an iterative scheme is hardly to be found and it is 
not clear whether the convergence (\ref{conv})
can still hold for $1/4\le \eta<1$ without any local assumption on $F_0$. 
Here we use the local assumption, i.e. the second condition in (\ref{local}); 
the method of proof of part (I) is similar to that in \cite{Lu2016}
where it is only concerned with $\eta=0$ (i.e. the hard sphere model). Here 
we are able to deal with all $0\le \eta<1$, but 
the condition (\ref{local}) leads to a `bad' consequence: when $\eta$ is closed to 1, the lower bound in
(\ref{local}) is very large and the temperature is very low:
$\overline{T}/\overline{T}_c<8\cdot 10^{-4} (b_0)^{2/3}(1-\eta)^{2}$ (see Lemma \ref{lemma3.4} below).
So this looks not so good, but part (II) shows further that if $\eta\ge 1$ and $\overline{T}/\overline{T}_c<1$, 
then there are even no strong convergence to equilibrium for all regular initial data.
Remember that this phenomenon of interaction potentials influencing  
the convergence to equilibrium has already happened in the classical space homogeneous Boltzmann equation:
for the hard potentials (with angular cutoff) and the hard sphere model, 
the rate of strong convergence to equilibrium is exponential and it does not depend
on any local information of initial data, whereas for Maxwellian molecules and for
soft potentials, the convergence rate depends heavily on the 
local information of initial data and can be arbitrarily slow  (see e.g.
\cite{CCL}, \cite{LM}).

\rm (2) Initial data $F_0$ that satisfy the condition (\ref{local})  exist
extensively. A general example is given in Section 3. 

\rm (3) Under the first condition (\ref{1.6*}) in the Assumption \ref{assp*}, the existence of a conservative measure-valued isotropic solution 
$F_t\in B_1^{+}({\mR}_{\ge 0})$ of Eq.(\ref{Equation1}) on tim-interval $[0,\infty)$ with
the any given initial datum $F_0\in B_1^{+}({\mR}_{\ge 0})$ has been proven in \cite{Lu2004}.
In fact in \cite{Lu2004} there is also an assumption of the strict positivity: $\int_{0}^{\pi/2}\cos(\theta)\Phi(V\cos(\theta),
V\sin(\theta)
)\sin(\theta){\rm d}\theta>0$ (for all $V=|{\bf v-v}_*|>0$)
which is only used to ensure the uniqueness of the equilibrium 
$F_{\rm be}$ (having the same mass and energy as $F_0$) and to prove the 
moment production; it is apparently not needed in the
proofs of weak stability and the existence theorems in \cite{Lu2004}. See also Remark 1.3 (3) in \cite{CL}.

\rm (4) Part (II) of Theorem\ref{theorem1.2} raises a question:  if $F_0(\{0\})>0$, i.e. if there is a seed of condensation at $t=0$,
does it hold $F_t(\{0\})\to F_{\rm be}(\{0\})\,(t\to\infty)$ ?

A method (see e.g.\cite{AN1},\cite{SH}) that may work for investigating this question is to split Eq.(\ref{Equation1}) into a equation system 
whose solution $(f,n_c)$ constitutes  a measure-valued solution $F_t$ of Eq.(\ref{Equation1}) with
${\rm d}F_t(x)=f(x,t)\sqrt{x}{\rm d}x+n_c(t)\dt(x){\rm d}x$ and $n_c(0)=F_0(\{0\})>0$. 
Experience shows however that the vanishing term $\min\{1, |{\bf v-v}_*|^{2\eta}\}$ with the larger order $\eta\ge 1$ 
in the upper bound (\ref{1.30}) may lead to the same difficulty in proving the propagation of condensation.
 \end{remark}
\vskip2mm

 The rest of the paper is organized as follows:
In Section 2 we introduce basic results on semi-strong convergence to equilibrium
that are used to prove the main result. In section 3, which is the core and technical part of the paper, we prove
some inequalities that are used to study the convergence and non-convergence to BEC. In section 4 we finish the proof of Theorem\ref{theorem1.2}.

\begin{center}\section{Semi-Strong Convergence to Equilibrium}\end{center}

Here we introduce a control relation between the semi-norm 
$\|F-F_{\rm be}\|_1^{\circ}$ and the 
entropy difference $S(F_{\rm be})-S(F)$ and then prove an improved version of the 
 entropy convergence.
\vskip2mm

\begin{lemma}\lb{lemma2.1} (\cite{CL}).Given $N>0, E>0$. Let $
 F \in{\cal B}_1^{+}({\mR}_{\ge 0})$ satisfy
 $N(F)=N, E(F)=E$. Then there is a constant $0<C<\infty$
depending only on $N, E$  such that
\be \fr{1}{C}\big(\|F-F_{\rm be}\|_1^{\circ}\big)^2
\le S(F_{\rm be})-S(F)\le C\big(\|F-F_{\rm be}\|_1^{\circ}\big)^{1/2}\lb{2.10}\ee
and if $\overline{T}/\overline{T}_c<1$, then
\be \big|F(\{0\})-F_{\rm be}(\{0\})\big|\le \|F-F_{\rm be}\|_{1}\le 2\big|F(\{0\})-F_{\rm be}(\{0\})\big|
+C(\|F-F_{\rm be}\|_{1}^{\circ})^{1/3}.\lb{2.11}\ee
Here
$$\|\mu\|_{1}^{\circ}=\int_{{\mR}_{\ge 0}}x{\rm d}|\mu|(x),\quad\|\mu\|_{1}=\int_{{\mR}_{\ge 0}}(1+x){\rm d}|\mu|(x),\quad   \mu\in {\cal B}_1({\mR}_{\ge 0}).$$
\end{lemma}
\vskip2mm

\begin{theorem}\label{theorem2.2} Let $B({\bf {\bf v-v}_*},\omega)$ satisfy Assumption \ref{assp}, let $F_0\in {\mathcal B}_1^{+}({\mathbb R}_{\ge 0})$ satisfy
$N(F_0)>0, E(F_0)>0$, and let $F_{\rm be}$ be
the unique Bose-Einstein distribution with the same mass $N=N(F_0)$ and energy $E=E(F_0)$. 
Let $\ld$ satisfy $\fr{1}{10\max\{2, \fr{4+2\eta}{3}\}} <\ld<\fr{1}{10\max\{2, \fr{4+2\eta}{3}\}-1}$.
Then there exists a conservative measure-valued isotropic solution $F_t\in {\cal B}_{1}^{+}({\mathbb R}_{\ge 0})$ of Eq.(\ref{Equation1}) such that 
\be 0\le S(F_{\rm be})-S(F_t)\le C(1+t)^{-\ld}\quad \forall\, t\ge 0\lb{e-conv}\ee
where  $C\in(0,\infty)$ depends only on $N,E, b_0, \eta$
and $\ld$.
\end{theorem}

\begin{proof} In case $0\le \eta<1/4$, the theorem has been proven in  \cite{CL} where we used the same condition $0\le \eta<1/4$ to cover both 
the entropy convergence and the strong convergence to equilibrium. But it is easily seen that 
for the entropy convergence itself, the restriction $0\le \eta<1/4$ can be relaxed to 
$0\le \eta<\infty$. Here we only give the key part of this improvement, for details 
see a renew version arXiv:1808.04038v4  of \cite{CL}.     

For any $0\le \eta<\infty$, let 
$\ld, p$ in the statement and the proof of Proposition 4.2 in \cite{CL}
be re-defined  by 
$$\fr{1}{10\max\{2, \fr{4+2\eta}{3}\}} <\ld<\fr{1}{10\max\{2, \fr{4+2\eta}{3}\}-1},\quad 
\max\big\{2, \fr{4+2\eta}{3}\big\}<p<\fr{1}{10}\big(1+\fr{1}{\ld}\big)$$ and as in \cite{CL} let 
$\dt=1-\ld (10p -1), q=p/(p-1).$ For convenience of reading we recall functions used in the derivation in the proof of Proposition 4.2 in \cite{CL}:
Let 
$B({\bf {\bf v-v}_*},\omega)$ satisfy Assumption \ref{assp} and let
  $\{B_K({\bf {\bf v-v}_*},\omega)\}_{K\in {\mN}}$ be the sequence of approximation of $B({\bf {\bf v-v}_*},\omega)$ 
 defined by (note that $|\la {\bf v-v}_*,\og\ra|=|{\bf v}-{\bf v}'|$)
 \beas&&
B_K({\bf {\bf v-v}_*},\omega)=
\fr{|{\bf v}-{\bf v}'|}{(4\pi)^2}\Phi_K(|{\bf v}-{\bf v}'|, |{\bf {\bf v}-{\bf v}_*'}|)\big|_{\og-{\rm rep.}},
\\
 &&{\rm with}\quad \Phi_K(r,\rho)=\min\big\{
\Phi(r,\rho),\,
(4\pi)^2K r\rho\big\}\eeas
and let
\bes&&
\bar{B}_K({\bf {\bf v-v}_*},\sg) =\fr{|{\bf v}-{\bf v}_*|}{2(4\pi)^2}\Phi_K(|{\bf v}-{\bf v}'|, |{\bf v}-{\bf v}_*'|)\big|_{\sg-{\rm rep.}}\nonumber \\
&&= \fr{|{\bf {\bf v-v}_*}|}{2(4\pi)^2}\min\big\{
\Phi(|{\bf v-v'}|, |{\bf {\bf v-v}_*'}|),\, (4\pi)^2K|{\bf v-v'}||{\bf {\bf v-v}_*'}|\big\}
\big|_{\sg-{\rm rep.}}.\lb{3.new3}\ees
Here $\og$-rep. and $\sg$-rep. mean that $({\bf v}', {\bf v}_*')$ is given by (\ref{colli}) and 
\be {\bf v}'=\fr{{\bf v}+{\bf v}_*}{2}+\fr{|{\bf v}-{\bf v}_*|\sg}{2},\quad 
{\bf v}_*'=\fr{{\bf v}+{\bf v}_*}{2}-\fr{|{\bf v}-{\bf v}_*|\sg}{2}, \qquad \sg\in{\mathbb S}^2\lb{sg-rep}\ee
respectively. It is well-known that the following equality holds for all suitable functions $\psi$: 
$$
\int_{{\bS}}B_K({\bf {\bf v-v}_*},\omega)\psi({\bf v}', {\bf v}_*')\big|_{\og-{\rm rep.}}{\rm d}\og=
\int_{{\bS}}\bar{B}_K({\bf {\bf v-v}_*},\sg)\psi({\bf v}', {\bf v}_*')\big|_{\sg-{\rm rep.}}{\rm d}\sg$$
which are often used to translate collision integrals with $\og$-representation into those with $\sg$-representation.
Let
\beas&& F^K({\bf v},t)=(1-e^{-t^{\dt}})f^K({\bf v},t)+
e^{-t^{\dt}}{\cal E}({\bf v}),\\
&&G({\bf v},t)=G^K({\bf v},t)=\fr{F^K({\bf v},t)}{1+F^K({\bf v},t)},\quad \quad ({\bf v},t)\in{\bR}\times [0,\infty).\eeas
where $f^K, {\cal E}$ are nonnegative integrable functions defined in the proof of Proposition 4.2 in \cite{CL}. Let
\beas&&  \Gm(a,b)=
\left\{\begin{array}{ll}
\displaystyle (a-b)\log\big(\fr{a}{b}\big)\,\qquad {\rm if}
\quad a>0, b>0\\
\displaystyle
\,\infty\qquad \quad\qquad \qquad \,\,  {\rm if}\quad
a>0=b\,\,\,{\rm or}\,\,\, a=0<b\\
\displaystyle
\,0\qquad \qquad \qquad \quad\,\,\,\,\,  {\rm if}\quad a=b=0
\end{array}\right.\\ \\
&&{\cal V}=\big\{({\bf v},{\bf v}_*,\sg)\in {\bRRS}\,\,\big|\,\, \Phi(|{\bf v-v'}|,
 |{\bf {\bf v-v}_*'}|)\ge (4\pi)^2K|{\bf v-v'}||{\bf {\bf v-v}_*'}|\big\},\eeas
${\cal V}^c={\bRRS}\setminus {\cal V}$ with constant $K\ge 1$.
Then the inequality (4.30) in \cite{CL} is 
rewritten (by H\"{o}lder inequality)
\beas&&{\cal D}_2(G(t)):=
\fr{1}{4}\int_{{\bRRS}}\la {\bf v-v}_*\ra^2
\Gm(G'G_*', GG_*)
{\rm d}{\bf v}{\rm d}{\bf v}_*{\rm d}\sg \\
&&\le\fr{1}{4}
\Big(\int_{{\cal V}}
\fr{\la {\bf v-v}_*\ra^{2q}}{(\bar{B}_K)^{{q}/{p}}}\Gm(G'G_*', GG_*)
{\rm d}{\bf v}{\rm d}{\bf v}_*{\rm d}\sg\Big)^{1/{q}} 
\Big(\int_{{\cal V}}\bar{B}_K\Gm(G'G_*', GG_*)
{\rm d}{\bf v}{\rm d}{\bf v}_*{\rm d}\sg\Big)^{1/{p}} 
\\
&&+ \fr{1}{4}
\Big(\int_{{\cal V}^c}
\fr{\la {\bf v-v}_*\ra^{2q}}{(\bar{B}_K)^{{q}/{p}}}\Gm(G'G_*', GG_*)
{\rm d}{\bf v}{\rm d}{\bf v}_*{\rm d}\sg\Big)^{1/{q}}\Big(
\int_{{\cal V}^c }
\bar{B}_K \Gm(G'G_*', GG_*)
{\rm d}{\bf v}{\rm d}{\bf v}_*{\rm d}\sg
\Big)^{1/{p}} \eeas
where 
$$\la{\bf u}\ra=\big(1+|{\bf u}|^2\big)^{1/2},\quad {\bf u}\in {\bR}.$$
The estimates for the three integrals $\int_{{\cal V}}
\fr{\la {\bf v-v}_*\ra^{2q}}{(\bar{B}_K)^{{q}/{p}}}\Gm(G'G_*', GG_*)
{\rm d}{\bf v}{\rm d}{\bf v}_*{\rm d}\sg$, 
$\int_{{\cal V}}\bar{B}_K\Gm(G'G_*', GG_*)
{\rm d}{\bf v}{\rm d}{\bf v}_*{\rm d}\sg,$ 
$\int_{{\cal V}^c}\bar{B}_K\Gm(G'G_*', GG_*)
{\rm d}{\bf v}{\rm d}{\bf v}_*{\rm d}\sg
$ are the same as 
in the proof of Proposition 4.2 in \cite{CL} (using (\ref{3.new3}) and the entropy dissipation inequality). For the fourth integral, using (\ref{3.new3}) and the lower bound in (\ref{1.29}) and using the inequality (4.31) in \cite{CL}
i.e.
$$GG_*>G'G_*'\,\Longrightarrow\,
\Gm(GG_*,G'G_*')\le C_{1} t^{\dt}GG_*\sqrt{1+|{\bf v}|^2+|{\bf v}_*|^2}$$
we deduce, as done in the proof of Proposition 4.2 in \cite{CL}, that
\beas&&\int_{{\cal V}^c}
\fr{\la {\bf v-v}_*\ra^{2q}}{(\bar{B}_K)^{{q}/{p}}}\Gm(G'G_*', GG_*)
{\rm d}{\bf v}{\rm d}{\bf v}_*{\rm d}\sg
\\
&&=\int_{\fr{1}{(4\pi)^2}\Phi(|{\bf v-v'}|, |{\bf {\bf v-v}_*'}|)< K|{\bf v-v'}||{\bf {\bf v-v}_*'}|}
\fr{\la {\bf v-v}_*\ra^{2q}}{(\bar{B}_K({\bf v-v}_*,\sg))^{q/p}}
\Gm(G'G_*',GG_*)
{\rm d}{\bf v}{\rm d}{\bf v}_*{\rm d}\sg\\
&&\le C_{2}\int_{{\bRRS}}\fr{\la {\bf v-v}_*\ra^{2q}}{
\big(|{\bf v-v}_*|\min\{1, |{\bf v-v}_*|^{2\eta}\}\big)^{q/p}}
\Gm(G'G_*',GG_*){\rm d}{\bf v}{\rm d}{\bf v}_*{\rm d}\sg\\
&&=2C_{1}\int_{{\bRRS}, GG_*>G'G_*'}\fr{\la {\bf v-v}_*\ra^{2q}}{
\big(|{\bf v-v}_*|\min\{1, |{\bf v-v}_*|^{2\eta}\}\big)^{q/p}}
\Gm(G'G_*',GG_*){\rm d}{\bf v}{\rm d}{\bf v}_*{\rm d}\sg\\
&&\le C_{3}t^{\dt}\int_{{\bRRS}}\fr{\la {\bf v-v}_*\ra^{2q}(\la {\bf v}\ra+
\la {\bf v-v}_*\ra)}{\big(|{\bf v-v}_*|\min\{1, |{\bf v-v}_*|^{2\eta}\}\big)^{q/p}}GG_*
{\rm d}{\bf v}{\rm d}{\bf v}_*{\rm d}\sg\\
&&\le
C_{4}t^{\dt}\big(\|G(t)\|_{L^1_1}+\|G(t)\|_{L^1}\|G(t)\|_{L^1_{2+q}}\big)\le C_{5}t^{\dt},\quad t\ge 1\eeas
where we have used the condition $(1+2\eta)\fr{q}{p}<3$ and $0<G(\cdot)<1$.
Note that all constants $0<C_i<\infty$ are
independent of $K$.   
 The rest of proof of (\ref{e-conv}) is completely the same as that of Proposition 4.2 in \cite{CL}.
 \end{proof}

  \begin{center}\section{Some Lemmas for Condensation}\end{center}

This section is a preparation for proving the main result Theorem \ref{theorem1.2}. 
Some inequalities with the Assumption \ref{assp} are improved version
of those in \cite{CL}: the restriction $0\le \eta<1/2$ is relaxed to 
the largest interval $0\le \eta<1$.  Some inequalities with the Assumption \ref{assp*} with $\eta\ge 1$ are
used to prove part (II) of Theorem\ref{theorem1.2}.
\vskip2mm

{\bf Notation}. To study local behavior of a measure $F\in {\cal B}^{+}({\mathbb R}_{\ge 0})$
near the origin, we introduce the following integrals. For
$p>0, \varepsilon>0, \alpha\ge 0$,   define
\beas&& N_{0,p}(F,\varepsilon)=\int_{{\mathbb R}_{\ge 0}}\big[\big(1-\frac{x}{\varepsilon}\big)_{+}\big]^p{\rm d}F(x)=\int_{[0, \varepsilon]}\big(1-\frac{x}{\varepsilon}\big)^p{\rm d}F(x),\\
&&
N_{\alpha, p}(F,\varepsilon)=\frac{1}{\varepsilon^{\alpha}}N_{0,p}(F,\varepsilon),\quad
\,\quad \underline{N}_{\alpha, p}(F,\varepsilon)=\inf_{0<\delta\le \varepsilon}N_{\alpha,p}(F,\delta),\\
&&
A_{0, p}(F,\varepsilon)=\int_{[0,\varepsilon]}\big(\frac{x}{\varepsilon}\big)^{p}
{\rm d}F(x),\quad A_{\alpha, p}(F,\varepsilon)=\fr{1}{\vep^{\alpha}}\int_{[0,\varepsilon]}
\big(\frac{x}{\varepsilon}\big)^{p}
{\rm d}F(x).\eeas

\begin{lemma}\lb{lemma3.2*} Let $B({\bf {\bf v-v}_*},\omega)$ be given by (\ref{kernel}),
 (\ref{Phi}), let $F\in B_1^{+}({\mR}_{\ge 0})$. Then
 
(I) For any $\vp\in C^{1,1}_b({\mR}_{\ge b0})$, it holds
 \bes \int_{{\mR}_{\ge 0}^3}{\cal K}[\vp]{\rm d}^3F\nonumber
&=&\int_{0< x<y\le z}\chi_{y,z}W(x,y,z)\Dt_{\rm sym}\vp(x,y,z){\rm d}^3F\nonumber\\
&+&
2\int_{0< x<y<z}\big(W(y,x,z)-W(x,y,z)\big)\Dt\vp(y,x,z){\rm d}^3F\nonumber
\\&+&\int_{0<y, z<x<y+z} W(x,y,z)
\Dt\vp(x,y,z){\rm d}^3F\nonumber \\
&+&F(\{0\})\int_{0<y\le z}\chi_{y,z}W(0,y,z)\Dt_{\rm sym}\vp(0,y,z){\rm d}^2F\nonumber\\
&+&2F(\{0\})\int_{0<y<z}\big(W(y,0,z)-W(0,y,z)\big)\Dt\vp(y,0,z){\rm d}^2F \label{dp}\ees
where $\Dt\vp(x,y,z)$ is defined 
in (\ref{diff}),
\bes && \Dt_{\rm sym}\vp(x,y,z)= \vp(z+y-x)+\vp(z+x-y)-2\vp(z),\quad 0\le x, y\le z, \dnumber \lb{sdp}\\
&& \chi_{y,z} =\left\{\begin{array}{ll} 2
 \,\,\,\,\qquad\quad {\rm if} \quad y< z,\\
1\,\,\,\,\qquad\quad {\rm if} \quad y=z.
\end{array}\right.\dnumber \lb{chi}\ees
(II) If $\Phi$ also satisfies (\ref{Phi*}), then for any convex function $\vp\in C_b^{1,1}({\mR}_{\ge 0})$, 
\be
\int_{{\mathbb R}_{\ge 0}^3}{\cal K}[\vp]{\rm d}^3F
\ge \int_{{\mathbb R}_{\ge 0}^3}{\cal K}_1[\vp]{\rm d}^3F\ge 0 \lb{3.4}
\ee
where
\bes&&{\cal K}_1[\vp](x,y,z)={\bf 1}_{\{0\le x<y\le z\}}
\chi_{y,z}W(x,y,z)\Dt_{\rm sym}\vp(x,y,z)\ge 0. \lb{3.5}\ees
\end{lemma}

\noindent\begin{proof} 
Part (II) is a result of Proposition 5.1(Convex-Positivity) in \cite{CL}.
The proof for part (I) is the same as that of the same proposition: 
using ${\cal K}[\vp](x,y,z)={\cal K}[\vp](x,z,y),$
${\cal K}[\vp](x,y,z)|_{x=y}=0, {\cal K}[\vp](x,y,z)|_{x=z}=0$, we have
\beas \int_{{\mathbb R}_{\ge 0}^3}{\cal K}[\vp]{\rm d}^3F
&=&\left(2\int_{0\le x<y<z}+2\int_{0\le y<x<z}+
\int_{0\le x<y=z}+\int_{0\le y, z<x}
\right)W(x,y,z)\Dt\vp(x,y,z){\rm d}^3F
\\
&:=& I_1+I_2+I_3+I_4.\eeas
Exchanging notations $x\leftrightarrow y$ for the integrand in $I_2$ gives
\beas
I_1+I_2 &=&
2\int_{0\le x<y<z}W(x,y,z)\Dt_{{\rm sym}}\vp(x,y,z){\rm d}^3F
\\
&+&
2\int_{0\le x<y<z}\big(W(y,x,z)-W(x,y,z)\big)\Dt\vp(y,x,z){\rm d}^3F.
 \eeas
Then using $\Dt\vp(x,y,z)|_{\{0\le x<y=z\}}=\Dt_{{\rm sym}}\vp(x,y,z)|_{\{0\le x<y=z\}}$
and recalling definition of $\chi_{y,z}$ in (\ref{chi}) 
we deduce 
\beas 
I_1+I_2+I_3 &=&
\int_{0\le x<y\le z}\chi_{y,z}W(x,y,z)\Dt_{{\rm sym}}\vp(x,y,z){\rm d}^3F
\\
&+&
2\int_{0\le x<y<z}\big(W(y,x,z)-W(x,y,z)\big)\Dt\vp(y,x,z){\rm d}^3F\\
&=&
\int_{0<x<y\le z}\chi_{y,z}W(x,y,z)\Dt_{{\rm sym}}\vp(x,y,z){\rm d}^3F
\\
&+&
2\int_{0<x<y<z}\big(W(y,x,z)-W(x,y,z)\big)\Dt\vp(y,x,z){\rm d}^3F\\
&+& F(\{0\})
\int_{0<y\le z}\chi_{y,z}W(0,y,z)\Dt_{{\rm sym}}\vp(0,y,z){\rm d}^2F
\\
&+&
2F(\{0\})\int_{0<y<z}\big(W(y,0,z)-W(0,y,z)\big)\Dt\vp(y,0,z){\rm d}^2F.
\eeas
Finally note that by definition of $W$ we have $W(x,y,z)=0$ for all $x\ge y+z$. So
$$I_4=\int_{0\le y, z<x<y+z}
W(x,y,z)\Dt\vp(x,y,z){\rm d}^3F
=\int_{0<y, z<x<y+z}
W(x,y,z)\Dt\vp(x,y,z){\rm d}^3F.$$
Combining these gives (\ref{dp}).
\end{proof}
\par

\begin{lemma}\lb{lemma3.2}(\cite{CL}) Let $B({\bf {\bf v-v}_*},\omega)$ be given by (\ref{kernel}),
 (\ref{Phi}), where $\Phi$ also satisfies (\ref{Phi*}).
Let $F_t\in {\cal B}_{1}^{+}({\mR}_{\ge 0})$ be  a conservative  measure-valued isotropic solution of Eq.(\ref{Equation1}) on
$[0, \infty)$ with the initial datum $F_0$
satisfying $N=N(F_0)>0, E=E(F_0)>0$.
Then with
$c=\sqrt{NE}$ we have: For any convex function $0\le \vp\in C^{1,1}_b({\mR}_{\ge 0})$
\be e^{ct}\int_{{\mR}_{\ge 0}}\vp{\rm d}F_t\ge e^{cs}
\int_{{\mR}_{\ge 0}}\vp{\rm d}F_s+\int_{s}^{t}e^{c\tau}{\rm d}\tau \int_{{\mR}_{\ge 0}^3}
{\cal K}[\vp]{\rm d}^3F_{\tau}\quad \forall\, 0\le s<t.\lb{3.12}\ee
And for any convex function $0\le \vp\in C_b({\mR}_{\ge 0})$, the function
$t\mapsto e^{ct}\int_{{\mR}_{\ge 0}}\vp{\rm d}F_t$ is non-decreasing on $[0,\infty)$, and thus
 for any $p\ge 1, \vep>0, \alpha>0$,  the functions
$t\mapsto e^{ct}N_{0,p}(F_t,\vep), 
t\mapsto e^{ct}\underline{N}_{\alpha,p}(F_t,\vep),$ and
$t\mapsto e^{ct}F_t(\{0\})$ are all non-decreasing on $[0,\infty)$.
\end{lemma}
\vskip2mm

In the following we will use a convention: in the set ${\mR}_{\ge 0}$ we define the arithmetic operation
\be \fr{b}{a}= \infty\quad {\rm if}\quad a=0<b.\lb{convention}\ee

\begin{lemma} \lb{lemma3.3} Let ${\cal J}[\vp], {\cal K}[\vp]$\
be defined in (\ref{JKW})-(\ref{Y}) where $\Phi(r,\rho)$ satisfies
 Assumption \ref{assp} with $0\le\eta<1$. Let $0\le \alpha<1-\eta, 
 p=\fr{3}{2}+\alpha, \beta=\fr{1-\alpha-\eta}{2},\vp_{\vep}(x)=[(1-x/\vep)_{+}]^2,
0<\vep\le 1$.
 
 (I) Let $F\in{\cal B}_1^{+}({\mathbb R}_{\ge 0})$.
Then
\bes && \int_{{\mathbb R}_{\ge 0}^3}{\cal K}[\vp_{\vep}]{\rm d}^3F
\ge \fr{b_0}{2}\underline{N}_{\alpha,2}(F,\vep)\big(A_{\beta,p}(F,\vep)\big)^2.
\dnumber\lb{KK}\ees

(II) Let
$F_t\in {\cal B}_{1}^{+}({\mR}_{\ge 0})$ be  a conservative  measure-valued isotropic solution of Eq.(\ref{Equation1}) on $[0, \infty)$ with the initial datum $F_0$
satisfying $N=N(F_0)>0, E=E(F_0)>0$.
Then for any $h>0$
\be F_{t}(\{0\})
\ge  e^{-2ch}N_{0,p}(F_{t-h},\vep)-\Big(\fr{2e^{-ch}N}{hb_0\underline{N}_{\alpha, 2}(F_{t-h},\vep)}\Big)^{1/2}
\Big(\fr{p}{\beta}\Big)^p\vep^{\beta}\quad \forall\, t\ge h.\lb{3.15}\ee
In particular for $\alpha=0$ we have
\be F_{t}(\{0\})\ge  e^{-2ch}N_{0,3/2}(F_{t-h},\vep)
-\Big(\fr{2e^{-c h}N}{hb_0F_{t-h}(\{0\})}\Big)^{1/2}
\Big(\fr{3}{1-\eta}\Big)^{3/2}\vep^{\fr{1-\eta}{2}}\quad \forall\, t\ge h.\lb{3.16}\ee
\end {lemma}

\begin{proof}  (I): We first prove that
\be W(x,y,z)\ge \fr{b_0}{2}\fr{z^{\eta}}{\sqrt{yz}}
\qquad \forall\, 0\le x<y\le z\le 1.\lb{W}\ee
First of all we recall (\ref{1.6}) that $\Phi(r,\rho)
\ge b_0\min\{1,\,(r^2+\rho^2)^{\eta}\}$ for all $r,\rho\ge 0$.
Take any $0\le x< y\le z\le 1$.

\noindent If $x=0$, then (recall (\ref{W2})) $W(0,y,z)=\fr{1}{\sqrt{yz}}
\Phi(\sqrt{2y}, \sqrt{2z}\,)\ge b_0\fr{z^{\eta}}{\sqrt{yz}}$.

\noindent Suppose $x>0$. Using the same proof in Lemma 5.3 in\cite{CL} 
we see that if $s\in [\sqrt{y}, \sqrt{x}+\sqrt{y}]$ then $
s+Y_*\ge \sqrt{z}$ and so using (\ref{1.6}) gives
$$\Phi(\sqrt{2}\,s, \sqrt{2}\,Y_*)
\ge b_0\min\{1,\, (s+Y_*)^{2\eta}\}\ge b_0z^{\eta}\qquad \forall\, \theta\in[0,2\pi]$$
and thus
$$W(x,y,z)
=\fr{1}{4\pi\sqrt{xyz}}
\int_{\sqrt{y}-\sqrt{x}}
^{\sqrt{x}+\sqrt{y}}{\rm d}s
\int_{0}^{2\pi}\Phi(\sqrt{2}s, \sqrt{2} Y_*){\rm d}\theta
\ge
\fr{b_0}{2} \fr{z^{\eta}}{\sqrt{yz}}.$$
This proves (\ref{W}).

It is easily seen that the function $x\mapsto \vp_{\vep}(x)$ is convex, belongs to
$C_b^{1,1}({\mathbb R}_{\ge 0})$, and holds the inequality
\be\Dt_{\rm sym}\vp_{\vep}(x,y,z)\ge \fr{(y-x)^2}{\vep^2} \quad {\rm for\,\,all}\,\,\, 0\le x\le y\le z\le \vep.\lb{vp}\ee
By Lemma \ref{lemma3.2} and (\ref{W}), (\ref{vp}) we have (recall that $0<\vep\le 1$)
\beas
\int_{{\mathbb R}_{\ge 0}^3}{\cal K}[\vp_{\vep}]{\rm d}^3F
&\ge &
\int_{0\le x\le y\le z\le \vep,\, y>0}\chi_{y,z}
\fr{b_0}{2}\fr{z^{\eta}}{\sqrt{yz}}
\fr{(y-x)^2}{\vep^2}{\rm d}F(x){\rm d}F(y){\rm d}F(z)
\\
&=&\fr{b_0}{2\vep^2}\int_{0<y\le z\le \vep}\chi_{y,z}
\fr{y^{\fr{3}{2}+\alpha}z^{\fr{3}{2}+\alpha}}{z^{2+\alpha-\eta}}
N_{\alpha,2}(F,y)
{\rm d}F(y){\rm d}F(z)
\\
&\ge &\fr{b_0}{2}\underline{N}_{\alpha,2}(F,\vep)\fr{1}{
\vep^{4+\alpha-\eta}}\int_{0<y\le z\le \vep}\chi_{y,z}
y^{\fr{3}{2}+\alpha}z^{\fr{3}{2}+\alpha}
{\rm d}F(y){\rm d}F(z)
\\
&= &\fr{b_0}{2}\underline{N}_{\alpha,2}(F,\vep)\fr{1}{\vep^{1-\alpha-\eta}}
\big(A_{0,p}(F,\vep)\big)^2=
\fr{b_0}{2}\underline{N}_{\alpha,2}(F,\vep)\big(A_{\beta,p}(F,\vep)\big)^2.
\eeas
(II):  By Lemma \ref{lemma3.2}, (\ref{KK}) and that $t\mapsto c^{ct}\underline{N}_{\alpha,2}(F_t,\vep)$ is non-decreasing, 
we have for any $h>0$ and $t\ge h$,
\beas e^{ct}N_{0,2}(F_{t},\vep)- e^{c(t-h)}N_{0,2}(F_{t-h},\vep)\ge \fr{b_0}{2}e^{c(t-h)}\underline{N}_{\alpha, 2}(F_{t-h},\vep)\int_{t-h}^{t}\big(A_{\beta,p}(F_{\tau},\vep)\big)^2 {\rm d}\tau.\eeas
Since $N_{0,2}(F_t,\vep)\le N$, this gives
\beas \fr{b_0}{2}\underline{N}_{\alpha, 2}(F_{t-h},\vep)\int_{t-h}^{t}\big(A_{\beta,p}(F_{\tau},\vep)\big)^2 {\rm d}\tau
\le e^{c h}N.\eeas
Then, using Cauchy-Schwarz inequality,
$$\fr{1}{h}\int_{t-h}^{t}A_{\beta, p}(F_\tau,\vep) {\rm d}\tau\le \Big(\fr{1}{h}\int_{t-h}^{t}
\big(A_{\beta, p}(F_\tau,\vep)\big)^2 {\rm d}\tau\Big)^{1/2}
\le\Big(\fr{2e^{c h}N}{hb_0\underline{N}_{\alpha, 2}(F_{t-h},\vep)}\Big)^{1/2} .$$
Note that according to the convention (\ref{convention}), this inequality
still holds when $\underline{N}_{\alpha, 2}(F_{t-h},\vep)=0$.
On the other hand, applying Lemma 2.3 in \cite{Lu2016} to the measure $F=F_{\tau}$ we have
$$ N_{0,p}(F_{\tau},\vep)\le F_{\tau}(\{0\})+\Big(\fr{p}{\beta}\vep^{\fr{\beta}{p}}\Big)^{p-1}
\int_{0}^{\vep}\vep_1^{-1+\fr{\beta}{p}} A_{\beta,p}(F_{\tau},\vep_1){\rm d}\vep_1
,\quad \vep>0.$$
Taking integration and using the fact that 
$0<\vep_1\le \vep\, \Longrightarrow\,\underline{N}_{\alpha,2}(F_{t-h},\vep_1)\ge \underline{N}_{\alpha,2}(F_{t-h},\vep)$ we have
\beas&&
\fr{1}{h}\int_{t-h}^{t}N_{0,p}(F_{\tau},\vep){\rm d}\tau\le \fr{1}{h}\int_{t-h}^{t}F_{\tau}(\{0\}){\rm d}\tau+\Big(\fr{p}{\beta}\vep^{\fr{\beta}{p}}\Big)^{p-1}
\int_{0}^{\vep}\vep_1^{-1+\fr{\beta}{p}} \fr{1}{h}\int_{t-h}^{t}A_{\beta,p}(F_{\tau},\vep_1){\rm d}\tau{\rm d}\vep_1
\\
&&\le \fr{1}{h}\int_{t-h}^{t}F_{\tau}(\{0\}){\rm d}\tau+\Big(\fr{p}{\beta}\vep^{\fr{\beta}{p}}\Big)^{p-1}
\int_{0}^{\vep}\vep_1^{-1+\fr{\beta}{p}}\Big(\fr{2e^{c h}N}{hb_0\underline{N}_{\alpha, 2}(F_{t-h},\vep_1)}\Big)^{1/2}{\rm d}\vep_1
\\
&&\le\fr{1}{h}
\int_{t-h}^{t}F_{\tau}(\{0\}){\rm d}\tau+
\Big(\fr{2e^{c h}N}{hb_0\underline{N}_{\alpha, 2}(F_{t-h},\vep)}\Big)^{1/2}
\Big(\fr{p}{\beta}\Big)^p\vep^{\beta}.
\eeas
Since
$t\mapsto e^{ct}N_{0,p}(F_t,\vep), t\mapsto e^{ct}F_t(\{0\})$ are  non-decreasing on $[0,\infty)$,
it follows that\\ 
$ e^{-ch}N_{0,p}(F_{t-h},\vep)\le N_{0,p}(F_{\tau},\vep),
F_{\tau}(\{0\})\le e^{ch} F_{t}(\{0\})$ for all $\tau\in [t-h,t]$
and so
$$ e^{-ch} N_{0,p}(F_{t-h},\vep)
\le e^{ch}F_{t}(\{0\})+
\Big(\fr{2e^{c h}N}{hb_0\underline{N}_{\alpha, 2}(F_{t-h},\vep)}\Big)^{1/2}
\Big(\fr{p}{\beta}\Big)^p\vep^{\beta}.$$
This gives (\ref{3.15})

Finally for the case $\alpha=0$, i.e. $\beta=(1-\eta)/2, p=3/2$,  we have $\underline{N}_{0, 2}(F_t,\vep)
=F_t(\{0\})$ and thus (\ref{3.16}) holds true.
\end{proof}
\vskip3mm

\begin{lemma} \lb{lemma3.4} Let $B({\bf {\bf v-v}_*},\omega)$ satisfy the Assumption \ref{assp} with $0\le \eta<1$. 
Let $F_0\in {\mathcal B}_1^{+}({\mathbb R}_{\ge 0})$ with
$N=N(F_0)>0, E=E(F_0)>0$ and suppose $F_0$ satisfy (\ref{local}).  
Let $F_t\in {\cal B}_{1}^{+}({\mR}_{\ge 0})$ be a conservative measure-valued isotropic solution of Eq.(\ref{Equation1}) on $[0,\infty)$ with initial datum $F_0$ satisfying $N(F_0)=N, E(F_0)=E$. Then $\overline{T}/\overline{T}_c< 8\cdot 10^{-4} (b_0)^{\fr{2}{3}}(1-\eta)^{2}$ and 
$$F_t(\{0\})\ge \fr{N}{8}\qquad \forall\, t\ge \fr{\log 2}{2\sqrt{NE}}.$$
\end{lemma}
\par
\noindent\begin{proof} Let 
$0\le \alpha<1-\eta, \beta=\fr{1}{2}(1-\eta-\alpha), p=\fr{3}{2}+\alpha$,
$c=\sqrt{NE}, h=\fr{\log 2}{2\sqrt{NE}}$, and let 
$$C_0=\fr{2^{9.5+\alpha}(4p)^{2\beta}}{(\log 2)b_0}
\Big(\fr{p}{\beta}\Big)^{2p}\Big(\fr{E}{N}\Big)^{\fr{1}{2}+2\beta}.$$ 
We prove the lemma with three steps.

\noindent{\bf Step1.} Prove that
\be \fr{E^{\fr{3}{2}-\eta}}{N^{\fr{5}{2}-\eta}}\le  \fr{(\log 2)b_0}{2^{9.5}6^{1-\eta}}
\Big(\fr{1-\eta}{3}\Big)^{3},\quad \frac{\overline{T}}{\overline {T}_c}<8\cdot 10^{-4} (b_0)^{\fr{2}{3}}(1-\eta)^{2}.\lb{E-eta}\ee
By the assumption (\ref{local}) we have
$$2^{\alpha}N\ge \fr{F_0([0, 1/2])}{(1/2)^{\alpha}}\ge \inf_{0<\vep\le 1/2}\fr{F_0([0,\vep])}{\vep^{\alpha}}\ge C_0$$
and so
\beas\fr{E^{\fr{3}{2}-\eta-\alpha}}{N^{\fr{5}{2}-\eta-\alpha}}=
\fr{1}{N}\Big(\fr{E}{N}\Big)^{\fr{1}{2}+2\beta}
\le  \fr{2^{\alpha}}{C_0}\Big(\fr{E}{N}\Big)^{\fr{1}{2}+2\beta}=\fr{(\log 2)b_0}{2^{9.5}(4p)^{2\beta}}
\Big(\fr{\beta}{p}\Big)^{2p},\qquad\quad \eeas
\bes \fr{E^{\fr{3}{2}-\eta}}{N^{\fr{5}{2}-\eta}}
=\fr{E^{\fr{3}{2}-\eta-\alpha}}{N^{\fr{5}{2}-\eta-\alpha}}\Big(\fr{E}{N}\Big)^{\alpha}
&\le & \fr{(\log 2)b_0}{2^{9.5}(4p)^{1-\eta-\alpha}}
\Big(\fr{1-\eta-\alpha}{3+2\alpha}\Big)^{3+2\alpha}\big(\fr{1}{4p}\big)^{\alpha}\nonumber
\\
&\le & \fr{(\log 2)b_0}{2^{9.5}(4p)^{1-\eta}}
\Big(\fr{1-\eta}{3}\Big)^{3}\le \fr{(\log 2)b_0}{2^{9.5}6^{1-\eta}}
\Big(\fr{1-\eta}{3}\Big)^{3}.\nonumber
\ees
And the same argument also gives
\bes
\fr{E^{\fr{3}{2}}}{N^{\fr{5}{2}}}=
\fr{E^{\fr{3}{2}-\eta}}{N^{\fr{5}{2}-\eta}}\Big(\fr{E}{N}\Big)^{\eta}
&\le & \fr{(\log 2)b_0}{2^{9.5}6^{1-\eta}}
\Big(\fr{1-\eta}{3}\Big)^{3}\Big(\fr{1}{4p}\Big)^{\eta}\nonumber\\
&= &\fr{(\log 2)b_0}{2^{10.5} 3^4}
(1-\eta)^3\Big(\fr{6}{4p}\Big)^{\eta}
\le \fr{(\log 2)b_0}{2^{10.5}3^4}(1-\eta)^3,\nonumber\ees
\beas&& 
\frac{\overline{T}}{\overline {T}_c}<2.273\frac{E}{N^{\fr{5}{3}}}
<2.273\Big(\fr{(\log 2)b_0}{2^{10.5}3^4}(1-\eta)^{3}
\Big)^{\fr{2}{3}}
<8\cdot 10^{-4} (b_0)^{\fr{2}{3}}(1-\eta)^{2}.
\eeas

\noindent{\bf Step2.} Prove that 
\be F_{nh}(\{0\})\ge \fr{N}{4},\quad n=1,2,3,...\,.\lb{NH}\ee
By the assumption on  $F_0$ we have
\beas \fr{
N_{0,2}(F_0,\vep)}{\vep^{\alpha}}
\ge \fr{1}{\vep^{\alpha}}
\int_{[0,\vep/2]}(1-\fr{x}{\vep})^2{\rm d}F_0(x)
\ge \fr{1}{2^{2+\alpha}}\fr{F_0([0,\vep/2])}{(\vep/2)^{\alpha}}\ge 
\fr{C_0}{2^{2+\alpha}}\eeas
for all $0<\vep\le 1.$
So
\beas&&  \underline{N}_{\alpha,2}(F_0, \vep)\ge \fr{C_0}{2^{2+\alpha}}\quad \forall\, 0<\vep\le 1.\eeas
Also using $\big((1-\fr{x}{\vep})_{+}\big)^p\ge 1-p\fr{x}{\vep}$ for all $x\ge 0$ gives
\beas&& 
N_{0,p}(F_0,\vep)
\ge N-\fr{p}{\vep}E.\eeas
Then using Lemma \ref{lemma3.3} we have with $c=\sqrt{NE}, h=\fr{\log 2}{2\sqrt{NE}}$ 
that
\beas&&F_{h}(\{0\})
\ge e^{-2ch}N_{0,p}(F_{0},\vep)-\Big(\fr{2e^{-ch}N}{hb_0\underline{N}_{\alpha, 2}(F_0,\vep)}\Big)^{1/2}
\Big(\fr{p}{\beta}\Big)^p\vep^{\beta}\\
&&\ge \fr{1}{2}\Big(N- \fr{pE}{\vep}\Big)-
\Big(\fr{4\cdot 2^{-\fr{1}{2}}N^{\fr{3}{2}}E^{\fr{1}{2}}}{(\log 2)b_0\fr{C_0}{2^{2+\alpha}}}\Big)^{1/2}
\Big(\fr{p}{\beta}\Big)^p\vep^{\beta}
\\
&&= \fr{N}{2}-\fr{pE}{2\vep}-
\Big(\fr{2^{3.5+\alpha} N^{\fr{3}{2}}E^{\fr{1}{2}}}{(\log 2)b_0 C_0}\Big)^{1/2}
\Big(\fr{p}{\beta}\Big)^p\vep^{\beta},\quad 0<\vep\le 1.
\eeas
{\bf Case1}:
$$\fr{pE}{2}<
\Big(\fr{2^{3.5+\alpha} N^{\fr{3}{2}}E^{\fr{1}{2}}}{(\log 2)b_0 C_0}\Big)^{1/2}
\Big(\fr{p}{\beta}\Big)^p.$$
In this case we take 
$$\vep=\vep_0=\Bigg(\fr{pE}{2
\Big(\fr{2^{3.5+\alpha} N^{\fr{3}{2}}E^{\fr{1}{2}}}{(\log 2)b_0 C_0}\Big)^{1/2}
\Big(\fr{p}{\beta}\Big)^p}\Bigg)^{\fr{1}{1+\beta}}. 
$$
Then $0<\vep_0<1$ and 
\beas&& \fr{pE}{2\vep_0}=
\Big(\fr{2^{3.5+\alpha} N^{\fr{3}{2}}E^{\fr{1}{2}}}{(\log 2)b_0 C_0}\Big)^{1/2}
\Big(\fr{p}{\beta}\Big)^p\vep_0^{\beta},\\
&&\fr{pE}{2\vep_0}+
\Big(\fr{2^{3.5+\alpha} N^{\fr{3}{2}}E^{\fr{1}{2}}}{(\log 2)b_0 C_0}\Big)^{1/2}
\Big(\fr{p}{\beta}\Big)^p\vep_0^{\beta}
=\fr{pE}{\vep_0}=pE\Bigg(\fr{2}{pE}
\Big(\fr{2^{3.5+\alpha} N^{\fr{3}{2}}E^{\fr{1}{2}}}{(\log 2)b_0 C_0}\Big)^{1/2}
\Big(\fr{p}{\beta}\Big)^p\Bigg)^{\fr{1}{1+\beta}}=\fr{N}{4}
\eeas
where the last equality is due to definition of  $C_0$.
In fact we compute
\beas&& pE
\bigg(\fr{2}{pE}
\Big(\fr{2^{3.5+\alpha} N^{\fr{3}{2}}E^{\fr{1}{2}}}{(\log 2)b_0 C_0}\Big)^{1/2}
\Big(\fr{p}{\beta}\Big)^p\bigg)^{\fr{1}{1+\beta}} =\fr{N}{4}\\
&&\Longleftrightarrow 
4p\fr{E}{N}
\bigg(\fr{2}{p}
\Big(\fr{N}{E}\Big)^{\fr{3}{4}}\Big(\fr{2^{3.5+\alpha}}{(\log 2)b_0 C_0}\Big)^{1/2}
\Big(\fr{p}{\beta}\Big)^p\bigg)^{\fr{1}{1+\beta}}
= 1\\
&&\Longleftrightarrow 
\fr{2^{9.5+\alpha}(4p)^{2\beta}}{(\log 2)b_0}
\Big(\fr{p}{\beta}\Big)^{2p}\Big(\fr{E}{N}\Big)^{\fr{1}{2}+2\beta}= C_0\\
\eeas
which does hold by definition of $C_0$. So the above equality holds true and so 
for the number $\vep=\vep_0$ we obtain  
\beas&&F_{h}(\{0\})
\ge  \fr{N}{2}-2\cdot \fr{pE}{2\vep_0}=\fr{N}{4}.
\eeas
{\bf Case2}:
$$\fr{pE}{2}\ge 
\Big(\fr{2^{3.5+\alpha} N^{\fr{3}{2}}E^{\fr{1}{2}}}{(\log 2)b_0 C_0}\Big)^{1/2}
\Big(\fr{p}{\beta}\Big)^p.$$
For this case we choose $\vep=1$ to get
\beas F_{h}(\{0\})
\ge  \fr{N}{2}-\fr{pE}{2}-
\Big(\fr{2^{3.5+\alpha} N^{\fr{3}{2}}E^{\fr{1}{2}}}{(\log 2)b_0 C_0}\Big)^{1/2}
\Big(\fr{p}{\beta}\Big)^p
\ge \fr{N}{2}- pE\ge \fr{N}{4}.
\eeas
So (\ref{NH}) holds for $n=1$. 

Suppose (\ref{NH}) holds for some $n\in {\mN}$. 
Then as shown above we have (with $c=\sqrt{NE}, h=\fr{\log 2}{2\sqrt{NE}}$)
\beas&&
 F_{(n+1)h}(\{0\})\ge  e^{-2ch}N_{0,3/2}(F_{nh},\vep)
-\Big(\fr{2e^{-ch}N}{hb_0F_{nh}(\{0\})}\Big)^{1/2}
\Big(\fr{3}{1-\eta}\Big)^{3/2}\vep^{\fr{1-\eta}{2}}\\
&&\ge 
\fr{1}{2}\big( N-\fr{3}{2\vep}E\big)
-\Big(\fr{2\cdot 2\cdot 2^{-\fr{1}{2}}N^{\fr{3}{2}}E^{\fr{1}{2}}}{(\log 2)b_0\fr{N}{4}}\Big)^{1/2}
\Big(\fr{3}{1-\eta}\Big)^{3/2}\vep^{\fr{1-\eta}{2}}
\\
&&= \fr{N}{2}-\fr{3}{4\vep}E
-\Big(\fr{2^{3.5}N^{\fr{1}{2}}E^{\fr{1}{2}}}{(\log 2)b_0}\Big)^{1/2}
\Big(\fr{3}{1-\eta}\Big)^{3/2}\vep^{\fr{1-\eta}{2}},\quad 0<\vep\le 1.
\eeas
{\bf Case3}: 
$$\fr{3}{4}E
<\Big(\fr{2^{3.5}N^{\fr{1}{2}}E^{\fr{1}{2}}}{(\log 2)b_0}\Big)^{1/2}
\Big(\fr{3}{1-\eta}\Big)^{3/2}.$$
In this case we choose 
$$\vep=\vep_1=\Bigg(\fr{3E}{4\Big(\fr{2^{3.5}N^{\fr{1}{2}}E^{\fr{1}{2}}}{(\log 2)b_0}\Big)^{1/2}
\Big(\fr{3}{1-\eta}\Big)^{3/2}}\Bigg)^{\fr{2}{3-\eta}}. 
$$
Then $0<\vep_1<1$,  
$$\fr{3}{4\vep_1}E=\Big(\fr{2^{3.5}N^{\fr{1}{2}}E^{\fr{1}{2}}}{(\log 2)b_0}\Big)^{1/2}
\Big(\fr{3}{1-\eta}\Big)^{3/2}\vep_1^{\fr{1-\eta}{2}}
$$
and so
\beas&& F_{(n+1)h}(\{0\})\ge \fr{N}{2}-2\cdot \fr{3E}{4\vep_1}
\\
&&=\fr{N}{2}-\fr{3E}{2}\bigg(\fr{4}{3E}\Big(\fr{2^{3.5}N^{\fr{1}{2}}E^{\fr{1}{2}}}{(\log 2)b_0}\Big)^{1/2}
\Big(\fr{3}{1-\eta}\Big)^{3/2}
\bigg)^{\fr{2}{3-\eta}}.\eeas
We then compute
$$\fr{N}{2}-\fr{3E}{2}\bigg(\fr{4}{3E}\Big(\fr{2^{3.5}N^{\fr{1}{2}}E^{\fr{1}{2}}}{(\log 2)b_0}\Big)^{1/2}
\Big(\fr{3}{1-\eta}\Big)^{3/2}
\bigg)^{\fr{2}{3-\eta}}\ge \fr{N}{4}$$
$\Longleftrightarrow $
$$ 6\fr{E}{N}\bigg(\fr{4}{3E}\Big(\fr{2^{3.5}N^{\fr{1}{2}}E^{\fr{1}{2}}}{(\log 2)b_0}\Big)^{1/2}
\Big(\fr{3}{1-\eta}\Big)^{3/2}
\bigg)^{\fr{2}{3-\eta}}\le 1$$
$\Longleftrightarrow $
$$ \fr{E^{\fr{3}{2}-\eta}}{N^{\fr{5}{2}-\eta}}\le \Big(\fr{1}{6}\Big)^{3-\eta}\bigg(\fr{3}{4}\Big(\fr{(\log 2)b_0}{2^{3.5}}\Big)^{1/2}
\Big(\fr{1-\eta}{3}\Big)^{3/2}
\bigg)^{2}=
\fr{(\log 2)b_0}{2^{9.5}6^{1-\eta}}
\Big(\fr{1-\eta}{3}\Big)^{3}.$$
From (\ref{E-eta}) one sees that the above inequality does hold true and
thus we obtain
$F_{(n+1)h}(\{0\})\ge \fr{N}{4}.$

\noindent {\bf Case4}:
$$\fr{3}{4}E
\ge \Big(\fr{2^{3.5}N^{\fr{1}{2}}E^{\fr{1}{2}}}{(\log 2)b_0}\Big)^{1/2}
\Big(\fr{3}{1-\eta}\Big)^{3/2}.$$
In this case, taking $\vep=1$ and using the assumption (\ref{local}) 
which implies $\fr{3}{2}E\le \fr{3}{2}\cdot \fr{N}{4p}<\fr{N}{4}$ we obtain
\beas
 F_{(n+1)h}(\{0\})
\ge \fr{N}{2}-\fr{3}{4}E
-\Big(\fr{2^{3.5}N^{\fr{1}{2}}E^{\fr{1}{2}}}{(\log 2)b_0}\Big)^{1/2}
\Big(\fr{3}{1-\eta}\Big)^{3/2}
\ge  \fr{N}{2}-\fr{3}{2}E\ge \fr{N}{4}.\eeas
This proves that (\ref{NH}) holds also for $n+1$ and thus it holds for all $n\in {\mN}.$
\vskip2mm

\noindent{\bf Step3.} For every $t\ge h$, there is an $n\in {\mN}$ such that $nh\le t<(n+1)h$. Since $t\mapsto e^{ct}F_t(\{0\})$ 
is non-decreasing, this gives  
$e^{ct}F_t(\{0\})\ge e^{c nh}F_{nh}(\{0\}) \ge e^{c nh}\fr{N}{N}$ and so 
$$ F_t(\{0\})\ge e^{-c (t-nh)}\fr{N}{4}=e^{-ch}\fr{N}{4}=\fr{N}{8}.$$
 \end{proof}
 \par
 
There are many examples of $F_0$ that satisfy the condition (\ref{local}). Here is a general example.  
\vskip2mm

{\bf Example.} 
Let $0<\alpha<1-\eta, p=\fr{3}{2}+\alpha$, and let $G_0, F_0\in {\cal B}_{1}^{+}({\mR}_{\ge 0})$ 
be given by for $0<R, \ld <\infty$
$${\rm d}G_0(x)=x^{\alpha-\fr{3}{2}}g(Rx)\sqrt{x}{\rm d}x=x^{\alpha-1}g(Rx){\rm d}x,\quad 
F_0=\ld G_0$$
where $g$ is a strictly positive Borel measurable function on ${\mR}_{\ge 0}$ satisfying  
$$\int_{0}^{\infty}(1+x)x^{\alpha-\fr{3}{2}}g(x)\sqrt{x}{\rm d}x<\infty.$$
Suppose in addition that $g$ is either decreasing on ${\mR}_{\ge 0}$ or 
is continuous ${\mR}_{\ge 0}$, so that in either case we have $\min\limits_{0\le x\le 1/2}g(Rx)>0$
for all $0<R<\infty$.
 Compute
\beas&& 
N(G_0)=\int_{{\mR}_{\ge 0}}{\rm d}G_0(x)=\int_{{\mR}_{\ge 0}}x^{\alpha-1}g(Rx){\rm d}x
=\Big(\fr{1}{R}\Big)^{\alpha}\int_{{\mR}_{\ge 0}}x^{\alpha-1}g(x){\rm d}x,\\
&&E(G_0)=\int_{{\mR}_{\ge 0}}x{\rm d}G_0(x)=\int_{{\mR}_{\ge 0}}x^{\alpha}g(Rx){\rm d}x
=\Big(\fr{1}{R}\Big)^{1+\alpha}\int_{{\mR}_{\ge 0}}x^{\alpha}g(x){\rm d}x,\\
&&\fr{E(F_0)}{N(F_0)}=\fr{E(G_0)}{N(G_0)}=\fr{1}{R}\fr{\int_{{\mR}_{\ge 0}}x^{\alpha}g(x){\rm d}x}{\int_{{\mR}_{\ge 0}}x^{\alpha-1}g(x){\rm d}x}.
\eeas
Choose $0<R_0<\infty$ large enough such that
$$\fr{E(F_0)}{N(F_0)}=\fr{E(G_0)}{N(G_0)}<\fr{1}{4p}$$
and let $c_0=\min\limits_{0\le x\le 1/2}g(R_0x)$.
Then $c_0>0$ and  
\beas&& G_0([0,\vep])=\int_{0}^{\vep}x^{\alpha-1}g(Rx){\rm d}x
\ge c_0\int_{0}^{\vep}x^{\alpha-1}{\rm d}x =\fr{c_0}{\alpha}\vep^{\alpha}\quad \forall\, 0<\vep\le 1/2,\\
&&\Longrightarrow \fr{G_0([0,\vep])}{\vep^{\alpha}}\ge \fr{c_0}{\alpha}\qquad \forall\, \vep\in (0, 1/2].\eeas
Let  $0<\ld<\infty$ be large enough such that
$$\ld \fr{c_0}{\alpha}
\ge \fr{2^{9.5+\alpha}(4p)^{2\beta}}{(\log 2)b_0}
\Big(\fr{p}{\beta}\Big)^{2p}\Big(\fr{E(G_0)}{N(G_0)}\Big)^{\fr{1}{2}+2\beta}.$$
Then, since $\fr{E(G_0)}{N(G_0)}=\fr{E(F_0)}{N(F_0)}$, 
\beas  \inf_{0<\vep\le 1/2}\fr{F_0([0,\vep])}{\vep^{\alpha}}=\ld \inf_{0<\vep\le 1/2}\fr{G_0([0,\vep])}{\vep^{\alpha}}
\ge \ld \fr{c_0}{\alpha}\ge 
\fr{2^{9.5+\alpha}(4p)^{2\beta}}{(\log 2)b_0}
\Big(\fr{p}{\beta}\Big)^{2p}\Big(\fr{E(F_0)}{N(F_0)}\Big)^{\fr{1}{2}+2\beta}
.\eeas
So $F_0,N=N(F_0),E=E(F_0)$ satisfy the condition (\ref{local}).
$\hfill\Box$
\vskip2mm

The following lemma will be used to prove part (II) of Theorem \ref{theorem1.2}.
\vskip2mm

\begin{lemma}\lb{lemma0} Let $B({\bf {\bf v-v}_*},\omega)$ satisfies the Assumption \ref{assp*} with 
$\eta\ge 1$.  Then 
	\bes &&W(x,y,z)\le \min\{1, 2^{4\eta}\max\{x,y,z\}^\eta\}W_H(x,y,z) \quad \forall x,y,z\ge 0  \lb{W01} \\
	&&W(x,y,z)-W(y,x,z)\le 8\sqrt{2}C_{\Phi}, \quad \forall\, 0\le x\le y\le z \dnumber\lb{W05} \ees
where $W_H$ is the function corresponding to the hard sphere model (see (\ref{hard})) and 
$$C_{\Phi}=\max\Big\{\sup_{(r,\rho)\in {\mR}_{\ge 0}^2}
\Big|\fr{\p^2}{\p\rho^2}\Phi(r,\rho)\Big|, \,\sup_{(r,\rho)\in {\mR}_{\ge 0}^2}\Big|\fr{\p^2}{\p r\p \rho}\Phi(r,\rho)\Big|
\Big\}.$$
\end{lemma}

\noindent\begin{proof} We first prove that
\be |\Phi(r, \rho_1)-\Phi(r, \rho_2)|\le C_{\Phi}(r+(\rho_1\vee \rho_2))|\rho_1-\rho_2|
\qquad \forall\, (r,\rho_1), (r,\rho_2)\in {\mR}_{\ge 0}^2.\lb{Phir}\ee
To do this we need only to prove that 
$|\fr{\p}{\p \rho}\Phi(r,\rho)|\le C_{\Phi}(r+\rho)$ for all $(r,\rho)\in {\mR}_{\ge 0}^2.$

In fact using the Assumption \ref{assp*} on $\Phi$ and $\eta\ge 1$ we 
have $\Phi(0,0)=0, \fr{\p}{\p \rho}\Phi(0,0)=0$ and so 
\beas&& \big|\fr{\p}{\p \rho}\Phi(r,\rho)\big|
\le \big|\fr{\p}{\p \rho}\Phi(r,\rho)-\fr{\p}{\p \rho}\Phi(r,0)\big|+
\big|\fr{\p}{\p \rho}\Phi(r,0)-\fr{\p}{\p \rho}\Phi(0,0)\big|
\\
&&\le \max\big\{\sup_{(r_1,\rho_1)\in {\mR}_{\ge 0}^2}
\big|\fr{\p^2}{\p \rho_1^2}\Phi(r_1,\rho_1)\big|, 
\sup_{(r_1,\rho_1)\in {\mR}_{\ge 0}^2}\big|\fr{\p^2}{\p r_1\p \rho_1}\Phi(r_1,\rho_1)\big|
\big\}
(r+\rho)\\
&&=C_{\Phi}(r+\rho),\qquad \forall\, (r,\rho)\in {\mR}_{\ge 0}^2.\eeas
Now by (\ref{Y}) and (\ref{KK0}) we have, for the case of $ |\sqrt{x}-\sqrt{y}|\vee |\sqrt{x_*}-\sqrt{z}| \le s\le (\sqrt{x}+\sqrt{y})\wedge (\sqrt{x_*}+\sqrt{z}),s>0$, $x_*>0$, that
	\beas&& s\le 2\max\{\sqrt{x},\sqrt{y},\sqrt{z}\}, \\
	&& Y_*\le \sqrt{z-\fr{(x-y+s^2)^2}{4s^2}
	}
	+\sqrt{x-\fr{(x-y+s^2)^2}{4s^2}
	}\
	\le 2\max\{\sqrt{x},\sqrt{y},\sqrt{z}\}.\eeas
	\noindent Combining with (\ref{1.6*}) gives
	\beas &&\Phi(\sqrt{2}s, \sqrt{2}Y_*)\le \min\{1,2^{4\eta}\max\{x,y,z\}^\eta\} \eeas
	and thus by (\ref{W1}) and (\ref{1.difference}) we have
	\beas W(x,y,z)
	&=&\fr{1}{4\pi\sqrt{xyz}}
	\int_{|\sqrt{x}-\sqrt{y}|\vee |\sqrt{x_*}-\sqrt{z}|}
	^{(\sqrt{x}+\sqrt{y})\wedge(\sqrt{x_*}+\sqrt{z})}{\rm d}s
	\int_{0}^{2\pi}\Phi(\sqrt{2}s, \sqrt{2} Y_*){\rm d}\theta
	\\
	&\le& \fr{1}{4\pi \sqrt{xyz}}\int_{|\sqrt{x}-\sqrt{y}|\vee |\sqrt{x_*}-\sqrt{z}|}
	^{(\sqrt{x}+\sqrt{y})\wedge(\sqrt{x_*}+\sqrt{z})}{\rm d}s
	\int_{0}^{2\pi}\min\{1,2^{4\eta}\max\{x,y,z\}^\eta\}{\rm d}\theta\\
	&=&\min\{1,2^{4\eta}\max\{x,y,z\}^\eta\}\fr{\min\{\sqrt{x},\sqrt{y},\sqrt{z},\sqrt{x_*}\}}{\sqrt{xyz}} \qquad {\rm if}\quad x_*xyz>0. \eeas
	This proves (\ref{W01}) for the case $x_*xyz>0$. The proof for other cases listed in
	(\ref{W2}) are similar and easier.
	\par
To  prove (\ref{W05}), we may assume that $0\le x<y\le z$.
\vskip2mm

\noindent {\bf Case1}: $x=0$. In this case we have by definition of $W$ and
(\ref{Phir}) that
\beas&& W(0, y,z)-W(y,0,z)=\fr{1}{\sqrt{yz}}\big(\Phi(\sqrt{2y},\sqrt{2z})-
\Phi(\sqrt{2y},\sqrt{2(z-y)})\big)\\
&&
\le \fr{1}{\sqrt{yz}}C_{\Phi}(\sqrt{2y}+\sqrt{2z}) (\sqrt{2z}-\sqrt{2(z-y)})
\le 4C_{\Phi}.
 \eeas

\noindent{\bf Case2}: $x>0$. We compute, for $0<x<y\le z$,
\beas&& W(x, y,z)-W(y,x,z)
\\
&&=\fr{1}{4\pi\sqrt{xyz}}
	\int_{|\sqrt{x}-\sqrt{y}|\vee |\sqrt{x_*}-\sqrt{z}|}
	^{(\sqrt{x}+\sqrt{y})\wedge(\sqrt{x_*}+\sqrt{z})}{\rm d}s
	\int_{0}^{2\pi}\Phi(\sqrt{2}s, \sqrt{2} Y_*){\rm d}\theta
\\
&&-\fr{1}{4\pi\sqrt{xyz}}\int_{|\sqrt{y}-\sqrt{x}|\vee |\sqrt{y_*}-\sqrt{z}|}^{(\sqrt{y}+\sqrt{x})\wedge(\sqrt{y_*}+\sqrt{z})}{\rm d}s\int_{0}^{2\pi}\Phi(\sqrt{2}s, \sqrt{2} Y_*^{\sharp}){\rm d}\theta
\eeas
where 
$y_*=(x+z-y)_{+}$, $Y_*=Y_*(x,y,z,s,\theta), 
Y_*^{\natural}=Y_*(y,x,z,s,\theta)$. It is easily shown that
$$(\sqrt{x}+\sqrt{y})\wedge(\sqrt{x_*}+\sqrt{z})=\sqrt{x}+\sqrt{y},\quad 
(\sqrt{y}+\sqrt{x})\wedge(\sqrt{y_*}+\sqrt{z})=\sqrt{x}+\sqrt{y},$$
$$|\sqrt{x}-\sqrt{y}|\vee |\sqrt{x_*}-\sqrt{z}|=\sqrt{y}-\sqrt{x},\quad 
|\sqrt{y}-\sqrt{x}|\vee |\sqrt{y_*}-\sqrt{z}|=\sqrt{y}-\sqrt{x}.$$
So
\bes&& W(x, y,z)-W(y,x,z) \nonumber 
\\
&&=\fr{1}{4\pi\sqrt{xyz}}
	\int_{\sqrt{y}-\sqrt{x}}
	^{\sqrt{x}+\sqrt{y}}{\rm d}s
	\int_{0}^{2\pi}
\big(\Phi(\sqrt{2}s,\sqrt{2}Y_*)-
\Phi(\sqrt{2}s,\sqrt{2}Y_*^{\sharp})\big){\rm d}\theta. \lb{3.22}
\ees
We need to prove that
\be 0\le Y_*-Y_*^{\sharp}\le \sqrt{2y},\quad Y_*\le 2\sqrt{z},\quad   \forall\, s\in (\sqrt{y}-\sqrt{x}, \sqrt{x}+\sqrt{y}),\forall\, \theta\in [0,2\pi].
\lb{YY}\ee
To do this, let
$$u=\fr{(x-y+s^2)^2}{4s^2},\quad w=\fr{(y-x+s^2)^2}{4s^2},\quad s\in (\sqrt{y}-\sqrt{x}, \sqrt{x}+\sqrt{y}).$$
By (\ref{KK0}) we have 
$$z-u\ge 0,\quad x-u\ge 0\qquad \forall\, s\in (\sqrt{y}-\sqrt{x}, \sqrt{x}+\sqrt{y}).$$
By calculation we see that $w-u=y-x$, i.e., $x-u=y-w$,
and so $w=u+y-x$,
$z-w=z-u-y+x=z-y+x-u\ge 0, y-w= y-u-y+x=x-u\ge 0$. So
$$Y_*=\big|\sqrt{z-u}+ e^{{\rm i}\theta}\sqrt{x-u}\big|=
\sqrt{z-u+x-u+2\sqrt{z-u}\sqrt{x-u}\cos(\theta)}$$
$$Y_*^{\sharp}=\big|\sqrt{z-w}+ e^{{\rm i}\theta}\sqrt{y-w}\big|=
\sqrt{z-w+y-w+2\sqrt{z-w}\sqrt{y-w}\cos(\theta)}.$$
We next prove that for all $s\in (\sqrt{y}-\sqrt{x}, \sqrt{x}+\sqrt{y})$ and all $\theta\in [0,2\pi]$ 
\be z-u+x-u+2\sqrt{z-u}\sqrt{x-u}\cos(\theta)\ge z-w+y-w+2\sqrt{z-w}\sqrt{y-w}\cos(\theta).\lb{YY2}\ee
To do this we denote 
$a=z-u, a_1=z-w, b= x-u=y-w$.  Then $a=z-u\ge z-w=a_1\ge 0$ because $w\ge u$. So
we compute 
\beas&& a=z-u\ge z-w=a_1\ge y-w=b,\\
&&
z-u+x-u+2\sqrt{z-u}\sqrt{x-u}\cos(\theta)-\big(z-w+y-w+2\sqrt{z-w}\sqrt{y-w}\cos(\theta)
\big)
\\
&&=(\sqrt{a}-\sqrt{a_1})\big(\sqrt{a}+\sqrt{a_1}-2\sqrt{b}\big)
\ge(\sqrt{a}-\sqrt{a_1})\big(\sqrt{b}+\sqrt{b}-2\sqrt{b}\big)=0.\eeas
Therefore (\ref{YY2}) holds true, and so $Y_*\ge Y_*^{\sharp}$ for all $s\in (\sqrt{y}-\sqrt{x}, \sqrt{x}+\sqrt{y})$ and all $\theta\in [0,2\pi]$.

Now using the inequality $0\le X_1\le X_2 \Longrightarrow 
\sqrt{X_2}-\sqrt{X_1}\le \sqrt{X_2-X_1}$ we have 
\beas&& 0\le  Y_*- Y_*^{\sharp}\\
&&\le \sqrt{ y-x+
2\sqrt{x-u}(\sqrt{z-u}-\sqrt{z-w})\cos(\theta)}
\\
&&\le \sqrt{ y-x+
2\sqrt{x-u}(\sqrt{z-u}-\sqrt{z-w})}\quad ({\rm because}\,\, u\le w \Longrightarrow 
z-u\ge z-w)\\
&&\le \sqrt{ y-x+
2\sqrt{x-u}\sqrt{w-u}}
=\sqrt{ y-x+
2\sqrt{x-u}\sqrt{y-x}}
\le \sqrt{2y-x-u}\le \sqrt{2y}.\eeas
Also we have 
$$ Y_*\le \sqrt{z-u}+\sqrt{x-u}\le \sqrt{z}+\sqrt{x}\le 2\sqrt{z}.$$
Thus (\ref{YY}) holds true.

From (\ref{YY}) and  (\ref{Phir}) we obtain
\beas&&
\Phi(\sqrt{2}s, \sqrt{2}Y_*)-
\Phi(\sqrt{2}s, \sqrt{2}Y_*^{\sharp})
\le C_{\Phi}(\sqrt{2}s+\sqrt{2}Y_*)\sqrt{2}(Y_*-Y_*^{\sharp})\\
&&
\le C_{\Phi} (\sqrt{2} 2\sqrt{y}+\sqrt{2} Y_*)\sqrt{2}(Y_*-Y_*^{\sharp})
\le C_{\Phi} 8\sqrt{2}\sqrt{yz}\\
&& \forall\, s\in (\sqrt{y}-\sqrt{x}, \sqrt{x}+\sqrt{y}),\forall\, \theta\in [0,2\pi].
\eeas
Then it follows from (\ref{3.22}) that
\beas&& W(x, y,z)-W(y,x,z)
\le \fr{1}{4\pi\sqrt{xyz}}2\sqrt{x}\cdot 2\pi \cdot 8\sqrt{2}C_{\Phi}\sqrt{yz}=8\sqrt{2}C_{\Phi}.
\eeas 
	This completes the proof of the lemma.
\end{proof}

\begin{center}\section{Proof of 
Theorem \ref{theorem1.2}}\end{center}

\par
Part(I): Let $F_t\in {\cal B}_{1}^{+}({\mR}_{\ge 0})$ be a conservative measure-valued isotropic solution of Eq.(\ref{Equation1}) on $[0,\infty)$ with initial datum $F_0$ satisfying $N(F_0)=N, E(F_0)=E$
obtained by Theorem\ref{theorem2.2} with $\fr{1}{20}<\ld<\fr{1}{19}$. 
Then
\be \|F_{\tau}-F_{\rm be}\|_{1}^{\circ}\le
C_1(1+t)^{-\ld/2}\qquad \forall\, t\ge 0.\lb{3.0}\ee
Here and below the constants $C, C_i\in(0,\infty), i=1,2,...,7$ depend only on $N, E, b_0,\eta$ and $\ld$. According to Lemma \ref{lemma2.1} and (\ref{3.0}), to prove part (I) of Theorem \ref{theorem1.2},
it needs only to prove that
\be \big|F_t(\{0\})-F_{\rm be}(\{0\})\big|\le  C(1+t)^{-\fr{(1-\eta)\ld}{2(4-\eta)}}\qquad \forall\, t\ge 0\lb{3.1}\ee
Using the inequality $0\le 1-[(1-\fr{x}{\vep})_{+}]^p\le \fr{p}{\vep}x\,(x\ge 0, \vep>0, 1<p<\infty)$ and 
using the conservation of mass we have with $p=\fr{3}{2}$ that
\beas&& 
\big|N_{0,3/2}(F_t,\vep)-N_{0,3/2}(F_{\rm be},\vep)\big|
\le \int_{{\mR}_{\ge 0}}\Big(1-\big[(1-\fr{x}{\vep})_{+}\big]^{3/2}\Big)
{\rm d}|F_t-F_{\rm be}|(x)\\
&&
\le \fr{3}{2\vep} \|F_{\tau}-F_{\rm be}\|_{1}^{\circ}
\le \fr{3}{2\vep}C_1(1+t)^{-\ld/2},\qquad t\ge 0.\eeas
Since
\beas&& N_{0,3/2}(F_{\rm be},\vep)\ge F_{\rm be}(\{0\}),\\
&&N_{0,3/2}(F_{\rm be},\vep)=F_{\rm be}(\{0\})+\int_{0}^{\vep}\big(1-\fr{x}{\vep}\big)^{\fr{3}{2}}
\fr{\sqrt{x}}{e^{\kappa x}-1}{\rm d}x\le F_{\rm be}(\{0\})+C_2 \vep^{1/2}
\eeas
it follows that
\beas&& N_{0,3/2}(F_t,\vep)\ge F_{\rm be}(\{0\})-\fr{3}{2\vep}C_1(1+t)^{-\ld/2},\\
&& F_t(\{0\})\le N_{0,3/2}(F_t,\vep)\le F_{\rm be}(\{0\})+C_2 \vep^{1/2}+\fr{3}{2\vep}C_1(1+t)^{-\ld/2}.\eeas
Taking $\vep=(1+t)^{-\ld/3}$ gives
\be F_t(\{0\})\le F_{\rm be}(\{0\})+C_3(1+t)^{-\ld/6},\qquad \forall\, t\ge 0.\lb{Uper}\ee
Let $h_0=\fr{\log 2}{2\sqrt{NE}}$. 
Then for every  $t\ge 1+h_0$,  taking 
$\vep=(1+t)^{-\fr{3\ld}{2(4-\eta)}}$ and $h=\vep^{\fr{1-\eta}{3}}$ and using (\ref{3.16}) in Lemma \ref{lemma3.3} 
and the above inequalities we 
obtain
\beas&&
F_{t}(\{0\})\ge e^{-2ch}N_{0,3/2}(F_{t-h},\vep)
-\Big(\fr{2 e^{-ch}N}{h b_0F_{t-h}(\{0\})}\Big)^{1/2}
\Big(\fr{3}{1-\eta}\Big)^{3/2}\vep^{\fr{1-\eta}{2}}\\
&&\ge (1-2ch)N_{0,3/2}(F_{t-h},\vep)
-\Big(\fr{16}{h b_0}\Big)^{1/2}
\Big(\fr{3}{1-\eta}\Big)^{3/2}\vep^{\fr{1-\eta}{2}}\\
&&\ge N_{0,\fr{3}{2}}(F_{t-h},\vep)-2ch N
-\Big(\fr{16}{b_0}\Big)^{1/2}
\Big(\fr{3}{1-\eta}\Big)^{3/2}\vep^{\fr{1-\eta}{2}} h^{-1/2}\\
&&\ge F_{\rm be}(\{0\})-\fr{3}{2\vep}C_1(1+t-h)^{-\ld/2}
-2c\vep^{\fr{1-\eta}{3}}N
-\Big(\fr{16}{b_0}\Big)^{1/2}
\Big(\fr{3}{1-\eta}\Big)^{3/2}\vep^{\fr{1-\eta}{3}}\\
&&\ge F_{\rm be}(\{0\})-\fr{C_4}{\vep}(1+t)^{-\ld/2}
-C_5\vep^{\fr{1-\eta}{3}}\ge F_{\rm be}(\{0\})-C_6(1+t)^{-\fr{(1-\eta)\ld}{2(4-\eta)}}.
\eeas
From this we obtain
$$ F_{t}(\{0\})\ge F_{\rm be}(\{0\})-C_7(1+t)^{-\fr{(1-\eta)\ld}{2(4-\eta)}}\qquad \forall\, t\ge 0.$$
This together with (\ref{Uper})
leads to (\ref{3.1}).
\vskip2mm

Part(II): Let $\varphi_{\vep}(x)=[(1-x/\vep)_+]^2$. 
By $W(x,y,z)\le W_H(x,y,z)$ (because $\Phi\le 1$) we have
\beas{\cal J}[\vp_{\vep}](y,z) &\le & \frac{1}{2}
\int_{0}^{y+z}W(x,y,z)(\vp_{\vep}(x)+\vp_{\vep}(y+z-x))
\sqrt{x}{\rm d}x\\
&\le &  \fr{1}{2}\int_{0}^{y+z}W_H(x,y,z)(\vp_{\vep}(x)+\vp_{\vep}(y+z-x))
\sqrt{x}{\rm d}x\\
&=&\int_{0}^{y+z}W_H(x,y,z)\vp_{\vep}(x)
\sqrt{x}{\rm d}x.\eeas
Combining this with the fact that $W_H(x,y,z)\sqrt{x}\le \sqrt{2}/\sqrt{y+z}$ for all $0<x<y+z$ and $\sup\limits_{r>0}\fr{1}{\sqrt{r}}\int_0^{r} \vp_{\vep}(x){\rm d}x\le \sqrt{\vep}$ we obtain
\beas\int_{\mR^2 \ge 0}{\cal J}[\vp_{\vep}]{\rm d}^2F_\tau&
\le&\int_{y,z\ge 0,y+z>0}\fr{\sqrt{2} }{\sqrt{y+z}}\int_0^{y+z} \vp_{\vep}(x){\rm d}x {\rm d}^2F_\tau\le\sqrt{2}\sqrt{\vep}N^2 .\eeas
For the cubic integral $\int_{{\mathbb R}_{\ge 0}^3}{\cal K}[\vp_{\vep}]{\rm d}^3F_{\tau}$, using part (I) of
Lemma \ref{lemma3.2*} we have
\bes \int_{{\mR}_{\ge 0}^3}{\cal K}[\vp_{\vep}]{\rm d}^3F_{\tau}\nonumber
&=&\int_{0< x<y\le z}\chi_{y,z}W(x,y,z)\Dt_{\rm sym}\vp_{\vep}(x,y,z){\rm d}^3F_{\tau}\nonumber\\
&+&
2\int_{0< x<y<z}\big(W(y,x,z)-W(x,y,z)\big)\Dt\vp_{\vep}(y,x,z){\rm d}^3F_{\tau}\nonumber
\\&+&\int_{0<y, z<x<y+z} W(x,y,z)
\Dt\vp_{\vep}(x,y,z){\rm d}^3F_{\tau}\nonumber \\
&+&F_{\tau}(\{0\})\int_{0<y\le z}\chi_{y,z}W(0,y,z)\Dt_{\rm sym}\vp_{\vep}(0,y,z){\rm d}^2F_{\tau}\nonumber\\
&+&2F_{\tau}(\{0\})\int_{0<y<z}\big(W(y,0,z)-W(0,y,z)\big)\Dt\vp_{\vep}(y,0,z){\rm d}^2F_{\tau}\nonumber\\
&:=& I_{1,\vep}(\tau)+I_{2,\vep}(\tau)+I_{3,\vep}(\tau)+I_{4,\vep}(\tau)+I_{5,\vep}(\tau). \label{dp}\ees
By Definition \ref{definition1.1} of measure-valued solutions we then obtain
\bes \int_{{\mathbb R}_{\ge 0}}\varphi_{\vep}{\rm d}F_t
&=&
\int_{{\mathbb R}_{\ge 0}}\varphi_{\vep}{\rm d}F_0
+\int_0^t {\rm d}\tau\int_{{\mathbb R}_{\ge 0}^2}{\cal J}[\vp_{\vep}]{\rm d}^2F_{\tau}+\int_0^t{\rm d}\tau \int_{{\mathbb R}_{\ge 0}^3}{\cal K}[\varphi_{\vep}]{\rm d}^3F_{\tau}\nonumber \\
&\le & \int_{{\mathbb R}_{\ge 0}}\varphi_{\vep}{\rm d}F_0+\sqrt{2}N^2\sqrt{\vep}t+
\int_0^t \sum_{j=1}^5I_{j,\vep}(\tau){\rm d}\tau,\quad t\ge 0.\lb{3.19}\ees
We need to prove that
\bes&& \limsup_{\vep \to 0^+}\int_0^t \big(I_{1,\vep}(\tau)+I_{2,\vep}(\tau)+I_{3,\vep}(\tau)\big){\rm d}\tau \le 0, \lb{K0}\\
&&\limsup_{\vep \to 0^+}\int_0^t \big(I_{4,\vep}(\tau)+I_{5,\vep}(\tau)\big){\rm d}\tau \le
a\int_0^t F_{\tau}(\{0\}){\rm d}\tau \dnumber \lb{II}\ees
where $a =48\sqrt{2} N^2.$
First from $\lim\limits_{\vep\to 0^+}\vp_{\vep}(x)=0$ for all $x>0$ it is easy to see that
\bes&&\lim_{\vep \to 0^+}W(x,y,z)\Dt_{\rm sym}\vp_{\vep}(x,y,z)=0 \quad \forall\,0<x<y\le z, \lb{W1*}\\
&&\lim_{\vep \to 0^+}(W(y,x,z)-W(x,y,z)\big)\Dt\vp_{\vep}(y,x,z)=0 \quad \forall\, 0<x<y\le z, \dnumber \lb{W2*}\\
&&\lim_{\vep \to 0^+}W(x,y,z)\Dt\vp_{\vep}(x,y,z)=0 \quad \forall\, 0<y,z< x<y+z. \dnumber \lb{W3*}
\ees
In order to use dominated convergence theorem to prove (\ref{K0}) and (\ref{II}) , it needs to prove
\bes &&0\le W(x,y,z)\Dt_{\rm sym}\vp_{\vep}(x,y,z) \le  16\sqrt{2} \quad \forall\,0\le x<y\le z,\lb{E1}\\
&&\max\big\{0, (W(y,x,z)-W(x,y,z)\big)\Dt\vp_{\vep}(y,x,z)\big\} \le 8\sqrt{2}C_{\Phi} \nonumber \\
&&\forall\,0\le x<y\le z,\dnumber \lb{E2} \\
&&0\le W(x,y,z)\Dt\vp_{\vep}(x,y,z)\le 16\sqrt{2} \quad \forall\, 0<y,z< x<y+z, \dnumber \lb{E3} \ees
(note that the case $x=0$ is included in (\ref{E1}), (\ref{E2})).
To do this we first use (\ref{W01}) and the definition of $W_H$ to get
\bes&& W(x,y,z)
\le \fr{\min\{1, (2^4z)^\eta\}}{\sqrt{yz}}\qquad \forall\, 0\le x<y\le z, \lb{Wx}\\
&&
\min\{1,(2^4z)^{\eta}\}\le 2^4z\qquad \forall\, z\ge 0\quad ({\rm because}\,\,\eta\ge 1 ).\dnumber \lb{W3}\ees
To prove (\ref{E1}), we need the inequality
$-\vp_{\vep}'(x)=\fr{2}{\vep}(1-\fr{x}{\vep})_{+}\le \fr{1}{2x}$ for all $x>0$.
Suppose first that $0\le x<y<z$ or $0<x<y\le z$. Then, by convexity of $\vp_{\vep}$,
$$0\le \Dt_{\rm sym}\vp_{\vep}(x,y,z)\le \vp_{\vep}(z+x-y)-\vp_{\vep}(z) \le -\vp_{\vep}'(z+x-y)(y-x)\le 
\fr{y-x}{2(z+x-y)}$$
and so using (\ref{Wx}),(\ref{W3}) we see that if $y\le z/2$ then
$$0\le W(x,y,z)\Dt_{\rm sym}\vp_{\vep}(x,y,z)\le \fr{2^4z}{\sqrt{yz}}\cdot\fr{y-x}{2(z-y+x)}\le
2^{4}$$ and if $y> z/2 $, we also have
$$0\le W(x,y,z)\Dt_{\rm sym}\vp_{\vep}(x,y,z)\le  W(x,y,z)\le \fr{\min\{1,(2^4z)^\eta\}}{\sqrt{yz}}\le
2^4\sqrt{2}. $$
Next suppose $x=0<y=z$.  Then using (\ref{Wx}),(\ref{W3}) again we have
$$0\le W(0,z,z)\Dt_{\rm sym}\vp_{\vep}(0,z,z)\le  W(0,z,z)
\le \fr{\min\{1, (2^4z)^\eta\}}{z}\le 2^4.$$
This proves (\ref{E1}).

To prove (\ref{E2}), let $0\le x<y\le z$.
	From (\ref{diff}) one sees that  $-1\le \Dt\vp_{\vep}(y,x,z)\le 0$, and so if $W(y,x,z)- W(x,y,z)\ge 0$, then  $\big(W(y,x,z)-W(x,y,z)\big)\Dt\vp_{\vep}(y,x,z)\le 0$. 
Suppose $W(y,x,z)-W(x,y,z)\le 0$. Then using  (\ref{W05}), (\ref{Wx}), and
(\ref{W3})  we get
\beas&& 0\le \big(W(y,x,z)-W(x,y,z)\big)\Dt\vp_{\vep}(y,x,z)\\
&&=
\big(W(x,y,z)-W(y,x,z)\big)(-\Dt\vp_{\vep}(y,x,z))
\\
&&\le W(x,y,z)-W(y,x,z)\le 8\sqrt{2}C_{\Phi}.\eeas
The inequality (\ref{E3}) is obvious: we have for all $0<y,z< x<y+z$ (using (\ref{W3}) again)
$$0\le W(x,y,z)\Dt\vp_{\vep}(x,y,z)\le W(x,y,z) \le
\fr{\min\{1, (2^4x)^\eta\}}{\sqrt{x}\sqrt{y\vee z}}\le 2^{4+1/2}.$$
We have proved (\ref{E1}),(\ref{E2}),(\ref{E3}), and thus (\ref{K0}),(\ref{II}) hold true.
From (\ref{dp}), (\ref{K0}), (\ref{II}) we then obtain 
$$\limsup_{\vep\to 0+}\int_{0}^t\sum_{j=1}^5 I_{j,\vep}(\tau) {\rm d}\tau 
\le a \int_{0}^tF_{\tau}(\{0\}) {\rm d}\tau,\quad t\in [0,\infty).$$
This together with $\lim\limits_{\vep\to 0+}\int_{{\mR}_{\ge 0}}\vp_{\vep}{\rm d}F_t 
=F_t(\{0\})\,(\forall\, t\ge 0$) and
(\ref{3.19}) gives
$$F_t(\{0\})\le  F_0(\{0\})+a\int_0^t F_{\tau}(\{0\}) {\rm d}\tau\qquad \forall\, t\ge 0$$
and we conclude (\ref{noBEC}) by Gronwall's lemma. 
Finally if $F_0(\{0\})=0$ and $\overline{T}/\overline{T}_c<1$, then 
$F_t(\{0\})=0$ for all $t\ge 0$ and 
$\|F_t-F_{\rm be}\|\ge |F_t-F_{\rm be}|(\{0\})=
|F_t(\{0\})-F_{\rm be}(\{0\})|=F_{\rm be}(\{0\})>0$ for  all $t\ge 0$.
$\hfill\Box$	
\\

{\bf Acknowledgment}.  C.S. is partially supported by the National Key R$\&$D
Program of China, Project Number 2021YFA1002800.
\\


\begin{thebibliography}{99}

\bibitem{AN1} Arkeryd, L.; Nouri, A.: Bose condensates in interaction with excitations: a kinetic model. 
Comm. Math. Phys. {\bf 310}, no. 3, 765–788  (2012).
 
\bibitem{AN2} Arkeryd, L.; Nouri, A.: Bose condensates in interaction with excitations: a two-component space-dependent model close to equilibrium. J. Stat. Phys. {\bf 160} , no. 1, 209–238  (2015).
    
\bibitem{weak-coupling} Benedetto, D., Pulvirenti, M., Castella, F., Esposito, R.: On
the weak-coupling limit for bosons and fermions. Math. Models Methods
Appl. Sci. {\bf 15}, 1811-1843  (2005).

\bibitem{Briant-Einav} Briant, M., Einav, A.: On the Cauchy problem for the
homogeneous Boltzmann-Nordheim equation for bosons: local existence, uniqueness
and creation of moments. J. Stat. Phys.  {\bf 163}, 1108-1156 (2016).

\bibitem{CL} Cai, S., Lu, X.: The spatially homogeneous Boltzmann equation for Bose-Einstein particles: rate of strong convergence to equilibrium. J. Stat. Phys. {\bf 175} , no. 2, 289-350 (2019).

\bibitem{CCL} Carlen, E.A., Carvalho, M.C., Lu, X.: On strong convergence to equilibrium for the Boltzmann
    equation with soft potentials. J. Stat. Phys. {\bf 135}, 681-736 (2009).

\bibitem {Chapman and Cowling} Chapman, S., Cowling, T.G.: {\it The Mathematical
Theory of Non-Uniform Gases}. Third Edition (Cambridge University
Press, 1970).


\bibitem{Do} Dolbeault,J.: Kinetic models and quantum effects: a modified Boltzmann equation for Fermi-Dirac particles. Arch. Rational Mech. Anal. {\bf 127}, no. 2, 101-131 (1994).


\bibitem{ESY} Erd\"{o}s, L., Salmhofer, M., Yau, H.-T.: On the
quantum Boltzmann equation. J. Stat. Phys. {\bf 116}, 367-380  (2004).

\bibitem{EMV0} Escobedo, M., Mischler, S., Valle, M.A.:
Homogeneous Boltzmann equation in quantum relativistic kinetic theory.
Electronic Journal of Differential Equations, Monograph, 4.
Southwest Texas State University, San Marcos, TX, 2003. 85 pp.

\bibitem {EV1} Escobedo, M., Vel\'{a}zquez, J.J.L.:
On the blow up and condensation of supercritical
solutions of the Nordheim equation for bosons.   Comm. Math. Phys.
{\bf 330}, 331-365 (2014).

\bibitem {EV2} Escobedo, M., Vel\'{a}zquez, J.J.L.: Finite time blow-up and
condensation for the bosonic Nordheim equation.
Invent. Math. {\bf 200}, 761-847  (2015).

\bibitem{LiLu} Li, W., Lu, X.: Global existence of solutions of the Boltzmann equation
for Bose-Einstein particles with anisotropic initial data.
 J. Funct. Anal. {\bf 276},  231-283 (2019).

\bibitem{Lions} Lions, P.-L.: Compactness in Boltzmann's equation via Fourier integral operators and applications. III. J. Math. Kyoto Univ. 34, no. 3, 539-584 (1994).

 \bibitem{Lu2000}Lu, X.: A modified Boltzmann equation for Bose-Einstein particles: isotropic solutions and long-time behavior. J. Statist. Phys. {\bf 98}, no. 5-6, 1335–1394  (2000).

\bibitem{Lu2004} Lu, X.: On isotropic distributional solutions to the Boltzmann
equation for Bose-Einstein particles. J. Statist. Phys. {\bf
116}, 1597-1649  (2004).

\bibitem{Lu2005} Lu, X.: The Boltzmann equation for Bose-Einstein particles: velocity concentration and
convergence to equilibrium. J. Statist. Phys. {\bf 119}, 1027-1067  (2005).

\bibitem{Lu2013} Lu, X.: The Boltzmann equation for Bose-Einstein particles:
condensation in finite time.  J. Stat. Phys. {\bf 150}, 1138-1176  (2013).

\bibitem{Lu2016} Lu, X.: Long time convergence of the Bose-Einstein condensation.
J. Stat. Phys. {\bf 162}, 652-670 (2016).


\bibitem{LM} Lu, X., Mouhot, C.: 
On measure solutions of the Boltzmann equation, Part II: Rate of convergence to equilibrium. J. Differential Equations {\bf 258}, no. 11, 3742–3810 (2015). 


\bibitem{LS} Lukkarinen, J., Spohn, H.:
Not to normal order--notes on the kinetic limit for weakly interacting quantum fluids.
 J. Stat. Phys. {\bf 134}, 1133-1172  (2009).

\bibitem {Nordheim} Nordheim, L.W.: On the kinetic methods in the new statistics
and its applications in the electron theory of conductivity.
Proc. Roy. Soc. London Ser. A  {\bf 119}, 689-698  (1928).

\bibitem {Nouri} Nouri, A.:  Bose-Einstein condensates at very low temperatures:
a mathematical result in the isotropic case.
Bull. Inst. Math. Acad. Sin. (N.S.) {\bf 2}, 649-666  (2007).


\bibitem {OW} Ouyang, Z.; Wu, L.: On the quantum Boltzmann equation near Maxwellian and vacuum. J. Differential Equations {\bf 316}, 471-551  (2022).

\bibitem {Rudin} Rudin, W: {\it Real and Complex Analysis}. Third edition. McGraw-Hill Book Co.,
New York, 1987. xiv+416 pp. ISBN: 0-07-054234-1 00A05 (26-01 30-01 46-01).

\bibitem {SH} Spohn, H.: Kinetics of the Bose-Einstein condensation.
Physica D {\bf 239}, 627-634  (2010).

\bibitem {Uehling and Uhlenbeck}Uehling, E.A., Uhlenbeck, G.E.:
Transport phenomena in Einstein-Bose and Fermi-Dirac gases, I, Phys. Rev.
{\bf 43}, 552-561 (1933).

\bibitem{Zhou} Zhou, Y.-L.: Global well-posedness of the quantum Boltzmann equation for bosons interacting via inverse power law potentials. Adv. Math. {\bf 430} (2023), Paper No. 109234, 118 pp.

\end{thebibliography}
\end{document}